\documentclass[a4paper, british, final]{amsart}

\usepackage[numbers]{natbib}
\usepackage{standalone}
\usepackage[final]{microtype}
\usepackage[obeyspaces, spaces, hyphens]{url}
\usepackage{amsmath}
\usepackage{amssymb}
\usepackage{thmtools}
\usepackage{amsthm}
\usepackage{mathtools}
\usepackage{mathrsfs}
\usepackage{stmaryrd}
\usepackage{centernot}
\usepackage[dvipsnames, svgnames, cmyk]{xcolor}
\newcommand\myshade{85}
\colorlet{mylinkcolor}{violet}
\colorlet{mycitecolor}{YellowOrange}
\colorlet{myurlcolor}{Aquamarine}
\usepackage{graphicx}         % Tool for importing images
\usepackage{tikz}             % Drawing tool
\usepackage{tikz-3dplot}
\usetikzlibrary{math}
\usetikzlibrary{calc}
\usetikzlibrary{intersections}
\usetikzlibrary{decorations.markings}
\usepackage{tikz-cd}
\usepackage{caption}
\usepackage{subcaption}
\usepackage{pgfplots}
\usepgfplotslibrary{polar}
\usepgflibrary{shapes.geometric}
\usetikzlibrary{calc}
\pgfplotsset{my style/.append style={axis x line=middle, axis y line= middle, xlabel={$n$}, ylabel={$d$}, axis equal ,xlabel style={below right},
  ylabel style={above right} }}
\pgfplotsset{compat=1.16}
\usepackage{enumitem}
\usepackage{xspace}  
\usepackage{textcomp}
\usepackage{multirow}
\usepackage{tablefootnote}
\usepackage{varioref}
\usepackage[pdfusetitle]{hyperref}
\hypersetup{final}
\usepackage[nameinlink, capitalize, noabbrev]{cleveref}
\hypersetup{
  linkcolor  = mylinkcolor!\myshade!black,
  citecolor  = mycitecolor!\myshade!black,
  urlcolor   = myurlcolor!\myshade!black,
  colorlinks = true,
}
\declaretheorem[style = plain, numberwithin = section]{theorem}
\declaretheorem[style = plain,      sibling = theorem]{corollary}
\declaretheorem[style = plain,      sibling = theorem]{lemma}
\declaretheorem[style = plain,      sibling = theorem]{proposition}

\declaretheorem[style = definition, sibling = theorem]{definition}

\declaretheorem[style = remark, sibling = theorem]{example}

\declaretheorem[style = remark,     sibling = theorem]{remark}

\DeclareMathOperator{\Spec}{Spec}
\DeclareMathOperator{\Sb}{sb}
\DeclareMathOperator{\SB}{SB}
\DeclareMathOperator{\LL}{\mathbb{L}}
\DeclareMathOperator{\Vol}{Vol}
\DeclareMathOperator{\Supp}{Supp}
\DeclareMathOperator{\ord}{ord}
\DeclareMathOperator{\Var}{Var}
\DeclareMathOperator{\Cl}{Cl}
\DeclareMathOperator{\Pic}{Pic}
\DeclareMathOperator{\Proj}{Proj}
\DeclareMathOperator{\Hom}{Hom}
\DeclareMathOperator{\width}{width}
\DeclareMathOperator{\Conv}{Conv}
\DeclareMathOperator{\codim}{codim}
\DeclareMathOperator{\htt}{ht}
\DeclareMathOperator{\FI}{FI}

\DeclarePairedDelimiter{\ip}{\langle}{\rangle}

\newcommand{\floor}[1]{\left\lfloor #1 \right\rfloor}
\newcommand{\ceil}[1]{\left\lceil #1 \right\rceil}
\renewcommand{\P}{\mathbb{P}}
\newcommand{\N}{\mathbb{N}}    % Natural numbers
\newcommand{\Z}{\mathbb{Z}}    % Integers
\newcommand{\Q}{\mathbb{Q}}    % Rational numbers
\newcommand{\R}{\mathbb{R}}    % Real numbers
\newcommand{\C}{\mathbb{C}}    % Complex numbers
\newcommand{\A}{\mathbb{A}}    % Affine space

\newcommand{\X}{\mathscr{X}}
\newcommand{\Y}{\mathscr{Y}}
\newcommand{\OO}{\mathcal{O}}
\newcommand{\Lb}{\mathcal{L}}
\newcommand{\PP}{\mathscr{P}}
\newcommand{\M}{\mathcal{M}}
\newcommand{\I}{\mathcal{I}}
\newcommand{\J}{\mathcal{J}}
\newcommand{\set}[1] {\left\lbrace #1 \right\rbrace}

\newcommand{\defeq}{\vcentcolon=}
\newcommand{\sett}[2] {\set{#1 : #2 }}
\newcommand{\settb}[2] {\set{#1 ~|~ #2 }}
\newcommand{\Zo}{Z^\circ}
\newcommand{\sbb}[1]{[#1]_{\Sb}}
\newcommand{\Volsb}{\Vol_{\Sb}}

\newcommand{\Kvar}[1] {\mathbf{K}_0(\Var_{#1})}
\newcommand{\SBB}[1]{\Z[\SB_{#1}]}
\newcommand{\Span}[1]{\text{Span}(#1)}

\newcommand{\Kn}{k (\!( t^{1/n} )\!)}
\newcommand{\Rn}{k\llbracket t^{1/n} \rrbracket}
\newcommand{\KK}{k (\!( t )\!)}

\newcommand{\Kollar}{Koll\'{a}r~}
\newcommand{\locdiv}[2]{\underline{\Gamma}_{#1}(\operatorname{Div}^+_{#2})}

\newcounter{nmdthmcnt}
\newenvironment{introthm}[2][]{\addtocounter{nmdthmcnt}{1}%
    \newtheorem*{nmdthm\roman{nmdthmcnt}}{#2}%
    \begin{nmdthm\roman{nmdthmcnt}}{#1}}{\end{nmdthm\roman{nmdthmcnt}}}

\usepackage[obeyFinal, textsize    = footnotesize, figwidth    = 0.99\linewidth]{todonotes}

\title{On Stable Rationality of Polytopes}
\author{Simen Westbye Moe}
\address{Imperial College, Department of Mathematics, South Kensington Campus, London
SW7 2AZ, UK}
\email{s.moe20@imperial.ac.uk}
\date{}
\begin{document}
\defcitealias{M2}{Macaulay2}
\begin{abstract}
    Nicaise--Ottem introduced the notion of (stably) rational polytopes and studied this using a combinatorial description of the motivic volume. In this framework, we ask whether being non-stably rational is preserved under inclusions. We prove this holds for a large class of polytopes, leading to a combinatorial strategy for studying stable rationality of hypersurfaces in toric varieties. As a result, we obtain new bounds for non-stably rational hypersurface in projective space, improving the ones given by Schreieder when the field has characteristic 0. We also obtain similar bounds for double covers of projective space and some new classes of non-stably rational varieties in products of projective space.
\end{abstract}
\maketitle
\setcounter{tocdepth}{1}
\section*{Introduction}
A complex variety $X$ is said to be \emph{stably rational} if there are integers $m,n\geq 0$ such that $X\times \P^n$ is birational to $\P^m$. Determining whether a variety is stably rational is a complex problem, but in the past 50 years, we have seen many new techniques applied with great success. Artin--Mumford showed that some special quartic double solids are non-stably rational using cohomological invariants \cite{AM72}. More than 40 years later, Voisin showed that very general quartic double solids are non-stably rational \cite{Voi15}. This was done by showing that having a decomposition of the diagonal is a stable birational invariant that specializes in mildly singular families. This degeneration technique has been generalized and applied in many other cases, for example, \cite{Tot16, Sch19a, CP16}.

Nicaise--Shinder introduced a different degeneration technique using a motivic obstruction to stable rationality \cite{NS19}. Let $K$ be the field of Puiseux series over an algebraically closed field $k$ of characteristic $0$. Then there is a map between rings of stable birational equivalence classes
\begin{align} 
\Volsb \colon \SBB{K} &\to \SBB{k} \nonumber \\ 
X&\mapsto \sum_{E\in S(\X)} (-1)^{\codim E}\sbb{E} \label{volsbIntroEq}
\end{align}
where $S(\X)$ is the set of strata in the special fiber of any strictly toroidal model for $X$ (see \cref{section1} for details). The map $\Volsb$ takes a stably rational (smooth and proper) $K$-scheme to the class $\sbb{\Spec{k}}$. From this, the strategy is as follows. We try to construct a degeneration of the varieties we are interested in such that $\Volsb$ takes the generic fiber to a class not equal to $\sbb{\Spec{k}}$, thus obstructing stable rationality of the generic fiber. Even if the special fiber contains a non-stably rational component, it must be ensured that it does not cancel with a term of opposite sign in \eqref{volsbIntroEq}. Avoiding this type of cancellation is the crucial property we study in this paper.

One way to construct degenerations is via toric geometry. From a Laurent polynomial $f\in k[x_1^{\pm 1},\dots, x_n^{\pm 1}]$ one can construct its Newton polytope $\Delta_f$. This is a convex lattice polytope and, in the associated projective toric variety, one obtains a canonical compactification $Z(f)$ of the zero set $\Zo(f)\subset \mathbb{G}_m^n$. Certain subdivisions of $\Delta_f$ induce degenerations of $Z(f)$, and a formula for \eqref{volsbIntroEq} in terms of a subdivision was calculated in \cite{NO20a} (see \cref{tropicalVolume}). In that formula, a stratum is of the form $\Zo(g)$ where $\Delta_g$ is a cell in the subdivision (but not in the boundary of $\Delta_f$). In this way, we try to show that $\Zo(f)$ is non-stably rational by finding $g$ such that $\Zo(g)$ is non-stably rational and $\Delta_g\subset \Delta_f$.

Nicaise--Ottem used the abovementioned strategy to show that a very general quartic fivefold is not stably rational \cite[Theorem 5.1]{NO20a}. The Newton polytope of a general quartic fivefold is $4\Delta_5$, the standard 5-simplex dilated by $4$. They identified a non-stably rational subpolytope of $4\Delta_5$ (corresponding to a double cover of $\P^4$ branched along a quartic threefold) and constructed a subdivision such that the obstruction in \eqref{volsbIntroEq} is nontrivial. We illustrate their argument in the case of a quartic surface.

Let $4\Delta_3$ be the dilated 3-simplex and consider the subdivision represented in \cref{fig:subdivofsimplexIntro}. It induces a degeneration $\X$ of a quartic surface into two isomorphic surfaces $S$ (corresponding to the blue and red part in \cref{fig:subdivofsimplexIntro}) that intersect in an elliptic curve $E$ (corresponding to the green part). This means that
\begin{equation*}
\Volsb{(\X_K)} = 2\sbb{S} - \sbb{E} \neq \sbb{\Spec{k}}.
\end{equation*}
Since the obstruction is nontrivial, a very general quartic surface is not stably rational.

\begin{figure}
    \centering
    \tdplotsetmaincoords{80}{415}
\begin{tikzpicture}[tdplot_main_coords]	
\draw[dashed] (0,0,0) -- (0,4,0);
\foreach \z in {0,...,4}{
    \tikzmath{\zz = 4-\z;}
    \foreach \y in {0,...,\zz}{
        \tikzmath{\yy = 4-\z-\y;}
        \foreach \x in {0,...,\yy}{
        \draw[fill=black] (\x,\y,\z) circle (1pt);
}}}
\fill [red, fill opacity=0.5] (0,0,0) -- (0,0,2)  -- (0,4,0) -- cycle;
\fill [red, fill opacity=0.5] (0,0,0) -- (0,4,0) -- (4,0,0) --cycle;
\fill [red, fill opacity=0.5] (0,0,0) -- (0,0,2) -- (4,0,0) --cycle;
\fill [blue, fill opacity=0.5] (0,0,2) -- (0,0,4) -- (4,0,0) --cycle;
\fill [blue, fill opacity=0.5] (0,0,4) -- (4,0,0) -- (0,4,0) --cycle;
\fill [blue, fill opacity=0.5] (0,0,2) -- (0,0,4) -- (0,4,0) --cycle;
\fill [green, fill opacity=0.5] (0,0,2) -- (0,4,0) -- (4,0,0) -- cycle;
\draw[dashed] (0,0,2) -- (0,4,0);
\draw (0,0,2) -- (4,0,0);
\draw (0,4,0) -- (4,0,0);
\draw (0,0,0) -- (4,0,0);
\draw (0,0,0) -- (0,0,4);
\draw (4,0,0) -- (0,0,4);
\draw (0,4,0) -- (0,0,4);
\draw[fill=black] (0,0,0) circle (1pt);
\draw[fill=black] (1,0,0) circle (1pt);
\draw[fill=black] (2,0,0) circle (1pt);
\draw[fill=black] (3,0,0) circle (1pt);
\draw[fill=black] (4,0,0) circle (1pt);
\draw[fill=black] (0,0,1) circle (1pt);
\draw[fill=black] (0,0,2) circle (1pt);
\draw[fill=black] (0,0,3) circle (1pt);
\draw[fill=black] (0,0,4) circle (1pt);
\draw[fill=black] (3,1,0) circle (1pt);
\draw[fill=black] (2,2,0) circle (1pt);
\draw[fill=black] (1,3,0) circle (1pt);
\draw[fill=black] (0,4,0) circle (1pt);
\draw[fill=black] (0,3,1) circle (1pt);
\draw[fill=black] (0,2,2) circle (1pt);
\draw[fill=black] (0,1,3) circle (1pt);
\draw[fill=black] (0,0,4) circle (1pt);
\draw[fill=black] (2,1,1) circle (1pt);
\draw[fill=black] (1,2,1) circle (1pt);
\draw[fill=black] (0,3,1) circle (1pt);
\draw[fill=black] (0,3,1) circle (1pt);
\draw[fill=black] (3,0,1) circle (1pt);
\draw[fill=black] (2,0,1) circle (1pt);
\draw[fill=black] (1,0,1) circle (1pt);
\draw[fill=black] (1,0,2) circle (1pt);
\draw[fill=black] (1,0,3) circle (1pt);
\draw[fill=black] (2,0,2) circle (1pt);
\draw[fill=black] (1,1,2) circle (1pt);
\end{tikzpicture}
    \caption{A subdivision of $4\Delta_3$}
    \label{fig:subdivofsimplexIntro}
\end{figure}
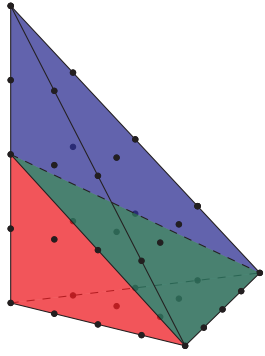

An essential aspect of this argument is controlling the stable birational types of all strata and the cancelation in \eqref{volsbIntroEq}. Above, this is achieved by constructing a degeneration where the two irreducible components are isomorphic. In general, we must ensure that two strata appearing with opposite signs in \eqref{volsbIntroEq} have distinct stable birational types if they are not stably rational. If this is possible, then any polytope that contains a non-stably rational polytope (not wholly contained in the boundary) is also not stably rational. The primary purpose of this paper is to study this property and its implications. We give unconditional results in dimensions at most four and partial results in higher dimensions. Our first result is the following.

\begin{introthm}{\cref{mainThmDim4}}
    Let $\delta\subset\Delta$ be lattice polytopes with $\dim\Delta \leq 4$, $\delta$ not stably rational, $\delta\not\subset \partial\Delta$. Then $\Delta$ is also not stably rational.
\end{introthm}

In an arbitrary dimension, we prove the following result. Condition $(M)$ (defined before \cref{monomialDegenThm}) is, in general, not satisfied but is typically easy to check. Even when it is not satisfied, we can often enlarge the polytope slightly so that it is.

\begin{introthm}{\cref{monomialDegenCor}}
    Let $\delta\subset \Delta$ be lattice polytopes with $\dim\delta=\dim\Delta$ such that $\delta$ is not stably rational and satisfies the condition $(M)$. Then $\Delta$ is also not stably rational.
\end{introthm}

We also prove a version that complements this, replacing condition $(M)$ with a smoothness condition; see \cref{smoothContainment}. We now have a combinatorial strategy for proving the non-stable rationality of polytopes by looking for non-stably rational subpolytopes satisfying any of the conditions mentioned above. Using this strategy, we give many new classes of non-stably rational varieties.

\Kollar showed that a very general hypersurface of degree $d\geq 2\ceil{(n+3)/3}$ is irrational \cite{Kol95}. This was done by constructing special examples admitting global differential forms after reduction modulo a prime $p$. Using these examples and the specialization technique initiated by Voisin, Totaro gave the first bound for stable rationality: A very general hypersurface of degree $d\geq 2\ceil{(n+2)/3}$ is not stably rational \cite{Tot16}. The state of the art is due to Schreieder. He proves that a very general hypersurface of degree $d\geq \log_2(n)+2$ is not stably rational \cite{Sch19a}. In fact, he proves that they are not even retract rational and gives results in positive characteristic. The proof is based on explicitly constructing varieties with no diagonal decomposition.

As a first application, we improve the bound given by Schreieder in characteristic 0. Note that our results are only valid in characteristic 0 and do not obstruct retract rationality.

\begin{introthm}{\cref{hypersurfacesTheorem}}
Let $N\geq 3$ and $X\subset \P^{N+1}$ be a very general hypersurface of degree $d\geq n+2$ where $N=n+r$ and $2^{n-1}-2 \leq r \leq 2^n-2+2^{n-2}(n-1)$. Then $X$ is non-stably rational.
\end{introthm}

The difference between our bound and the one in \cite{Sch19a} grows logarithmically with $n$ (see \cref{degreeRemark}). Schreieder gives examples of special hypersurfaces in all dimensions that are non-stably rational. These have small Newton polytopes compared to a general hypersurface of the same degree, and we use this together with \cref{monomialDegenCor} to obtain the new bound. For example, \cref{hypersurfacesTheorem} implies that quintics in $\P^{N+1}$ are non-stably rational if $N\leq 13$, and sextics are non-stably rational if $N\leq 30$. The cases where $N\leq 9$ and $N\leq 18$, respectively, are covered in \cite{Sch19a}.

Using the same methods, we give the following improvement on double covers of projective space.

\begin{introthm}{\cref{doubleCoverThm}}
Let $N\geq 3$ and write $N=n+r$ for $2^{n-1}-2\leq r\leq 2^{n}-2 + 2^{n-2}(n-1)-\floor{\frac{n}{2}}$. Then a double cover $X\to \P^N$ branched along a very general hypersurface of even degree $d\geq 2\ceil{\frac{n}{2}}+2$ is not stably rational.
\end{introthm}
This follows quite directly from \cref{hypersurfacesTheorem} viewed in terms of Newton polytopes. See the proof of \cref{doubleCoverThm}.

We also show the following classes of non-stably rational hypersurfaces in products of projective spaces.

\begin{introthm}{\cref{23Divisor}}
A very general hypersurface of bidegree $(2,3)$ in $\P^3\times \P^4$ or $(3,3)$ in $\P^4\times \P^4$ is not stably rational.
\end{introthm}

In \cref{section1}, we briefly describe the motivic volume and strictly toroidal models as in \cite{NS19, NO20a, NO20b}.

In \cref{section2}, we give an overview of some of the geometric invariants of $Z(f)$ that can be described in terms of $\Delta_f$. We recall the combinatorial formula from \cite{NO20a} that computes $\Volsb$ in terms of a subdivision. We also collect some examples of \emph{empty simplices} to illustrate the geometry of these from the point of view of rationality. These are polytopes where the framework of \cite{NO20a} cannot be applied and obstructs generalizing \cref{mainThmDim4} to any dimension.

In \cref{section3}, we study the variation of stable birational types in linear systems on toric varieties. We aim to use this to control the cancelation of terms in \eqref{volsbIntroEq}. We introduce certain classes of polytopes and prove that the stable birational type in the associated linear systems varies strongly. We prove this by constructing special degenerations not given by subdivisions of polytopes. Because of this, extra work is required to show they are strictly toroidal. This is established by proving a general result on strictly toroidal schemes in the language of logarithmic geometry, proved in \cref{appendixA}.

We prove \cref{mainThmDim4} in \cref{section4} by analyzing the types of surfaces that arise as ample divisors in toric threefolds. More specifically, we show that they cannot be birationally ruled over curves of positive genus.

\cref{section5} is devoted to applications. We apply the methods developed in \cref{section3} to hypersurfaces in projective space and double covers of projective space. We study the Newton polytopes associated with special hypersurfaces constructed in \cite{Sch19a}. We end by showing that very general hypersurfaces of bidegree $(2,3)$ and $(3,3)$ in $\P^3\times\P^4$ and $\P^4\times \P^4$ respectively are non-stably rational.

\subsection*{Notation}
We fix $k$ an algebraically closed field of characteristic 0, and let $F$ denote an arbitrary field of characteristic 0. By a variety over $F$, we mean a reduced separated $F$-scheme of finite type.

Let $K_0=k(\!(t)\!)$ denote the field of Laurent polynomials in $t$ and $R=k[\![t]\!]$ its discrete valuation ring. We let $K=\cup_{n\geq 1} \Kn$ be the field of Puiseux series, and we have a valuation

\begin{align*}
    \nu\colon K^\times \to \Q,~~ \sum_{k\in \Z} c_k t^{k/n} 
    \mapsto \min\settb{k/n}{c_k \neq 0}
\end{align*}
with valuation ring $R=\cup_{n\geq 1} \Rn$.

\subsection*{Acknowledgements}
I want to thank Johannes Nicaise and John Christian Ottem for their continuous support and guidance throughout this project. This research is funded by Aker Scholarship.

\section{Rationality and the motivic volume} \label{section1}
\subsection{Rationality of Varieties}
There are many measures of the birational complexity of a variety. We will focus mainly on \emph{ stable rationality}.
\begin{definition}
    Let $X$ be a variety over $F$ of dimension $n$. Then $X$ is said to be
    \begin{enumerate}
        \item[(1)] \emph{Rational} if it is birational to $\P^n_F$.
        \item[(2)] \emph{Stably rational} if there is an integer $\ell\geq 0$ such that $X\times \P^\ell_F$ is birational to $\P^{\ell+n}_F$.
        \item[(3)] Two varieties $X,Y$ are \emph{stably birational} if there are integers $r,s\geq 0$ such that $X\times_F \P^r_F$ is birational to $Y\times_F \P^{s}_F$.
        \end{enumerate}
        \end{definition}

Given an arbitrary algebraic variety, it is not easy to establish what kind of rationality property it has. An approach successfully applied in recent years is constructing special varieties $X$ with an obstruction to (stable) rationality and using specialization to spread the result in smooth proper families containing $X$. This is made precise by the following result.

\begin{theorem}[{\cite[Theorem 4.1.4]{NS19}}] \label{veryGeneralCor}
    Let $S$ be a Noetherian $\Q$-scheme, and let $f\colon X\to S$, $g\colon Y\to S$ be proper smooth morphisms. Then the set
    \begin{equation*}
    \settb{s\in S}{X\times_S \overline{s}~\text{is stably birational to}~Y\times_S \overline{s} }
    \end{equation*}
    where $\overline{s}$ is any geometric point based at $s$, is a countable union of closed sets.
\end{theorem}

For a smooth proper morphism $f\colon X\to S$, we say that a property holds for a general fiber if it holds for fibers over a (nonempty) Zariski open set on $S$. Similarly, a property holds for a very general fiber if it holds for fibers over a (nonempty) countable intersection of Zariski open sets on $S$. \cref{veryGeneralCor} implies that if $f$ has a non-stably rational fiber, then a very general fiber is also not stably rational.

Another implication of \cref{veryGeneralCor} is that the set of fibers of $f\colon X\to S$ that is stably birational to a fixed variety $Y$ is a countable union of Zariski closed sets on $S$. This means either the stable birational type of the fibers of $f$ is constant, or two very general fibers are not stably birational. In the latter case, we say that $f$ has \emph{variation of stable birational type} \cite{SV19, Sch19b}.

\begin{definition}
    Let $S$ be a Noetherian $\Q$-scheme, and $\pi\colon U\to S$ be a smooth proper morphism. We say that $\pi$ (or simply $S$ if $\pi$ is clear from context) has \emph{variation of stable birational type} (\emph{variation} for short) if for any variety $X$ the set
    \begin{align*}
        \operatorname{R}(X,\pi) = \settb{s\in S}{U\times_S\overline{s}~\text{is stably birational to}~X}
    \end{align*}
    where $\overline{s}$ is a geometric point over $s$, is a countable union of proper closed sets of $S$. 
\end{definition}

\subsection{Strictly Toroidal Models}
The stable birational invariant introduced in this section can be computed on strictly toroidal models. To describe these, we first give a few preliminaries on monoids (see \cite[Section 1.1.3]{Ogu18} for details). 

Let $M$ be a monoid (written additively) with identity $0$ and denote by $M^{\operatorname{gp}}$ the groupification of $M$. We say that $M$ is
\begin{itemize}
    \item \emph{integral} if for all $a,b,c\in M$ we have $a+c=b+c\iff a=b$,
    \item \emph{saturated} if for $a\in M^{\operatorname{gp}}$ and $n\in \Z_{>0}$ we have $na\in M \iff a\in M$,
    \item \emph{sharp} if the only invertible element in $M$ is $0$,
    \item \emph{fine} if it is integral and finitely generated.
\end{itemize}
\noindent
A monoid is \emph{toric} if fine, saturated, and sharp. These are the monoids of lattice points of finitely generated rational polyhedral cones.

For a monoid $M$, let $F[M]$ denote the monoid algebra, the free abelian group on elements $\chi^m, m\in M$ with multiplication $\chi^m\cdot \chi^{m'}=\chi^{m+m'}$.

\begin{definition}
    A flat separated $R$-scheme $\X$ of finite presentation is \emph{strictly toroidal} if there is a cover of open sets $\mathscr{U}\subset \X$ with a smooth morphism
    \begin{align*}
        \mathscr{U}\to \Spec{R[M]/(\chi^m-t^q)}
    \end{align*}
    such that $M$ is a toric monoid, $k[M]/(\chi^m)$ is reduced, and $q\in \Q_{> 0}$. If $X$ is a $K$-scheme, then a strictly toroidal model for $X$ is a strictly toroidal $R$-scheme $\X$ with an isomorphism $\X_K\simeq X$.
\end{definition}

\begin{example}
Let $\X$ be a flat proper $R$-scheme with $\X_K$ smooth over $K$. Then $\X$ is \emph{strictly semi-stable} if it can be covered by open sets $\mathscr{U}\subset \X$ with a smooth morphism
\begin{equation*}
\mathscr{U}\to R[x_1,\dots, x_n]/(x_1\dots x_r-t^q)
\end{equation*}
for some $r\leq n$. This is strictly toroidal by taking $M=\N^n$ and $m=e_1+\dots +e_r$ where $e_i$ denotes the standard basis on $\N^n$.
\end{example}

For a strictly toroidal scheme $\X$, write $S(\X)$ for the set of strata of $\X_k$, meaning connected components of an intersection of irreducible components in $\X_k$.

\subsection{The Motivic Volume}
The motivic volume as a stable birational invariant was first introduced in \cite{NS19}. It associates to any smooth proper $K$-scheme of finite type a formal sum of stable birational types of $k$-schemes and to any stably rational $K$-scheme the class of $\Spec{k}$. The arithmetic takes place in the ring of stable birational types.

Let $\SBB{F}$ denote the free abelian group generated by stable birational equivalence classes $\sbb{X}$ for irreducible $F$-varieties $X$. For a (possibly reducible) variety $X$ and a decomposition $X=X_1\cup\dots\cup X_r$ into irreducible components, we let $\sbb{X}=\sbb{X_1}+\dots + \sbb{X_r}$. We endow $\SBB{F}$ with a ring structure given by $\sbb{X}\cdot \sbb{Y}=\sbb{X\times_F Y}$.

\begin{theorem}[{\cite[Theorem 3.3.2]{NO20b}}] \label{motivicVolumeThm}
There is a unique ring map
\begin{equation*}
    \Volsb \colon \SBB{K} \to \SBB{k}
\end{equation*}
such that for a strictly toroidal $R$-scheme $\X$ with generic fiber $\X_K$ we have
\begin{equation*}
    \Volsb(\sbb{\X_K}) = \sum_{E\in S(\X)} (-1)^{\codim(E)} \sbb{E}
\end{equation*}
where the codimension of $E$ is taken in the special fiber $\X_k$. The map $\Volsb$ is called the \emph{stable birational volume}. We typically write $\Volsb(X)$ for $\Volsb(\sbb{X})$.
\end{theorem}

\begin{corollary}
    Let $X$ be a smooth proper $R$-scheme. If $\Volsb(X)\neq \sbb{\Spec{k}}$ then $X$ is not stably rational over $K$.
\end{corollary}
\begin{proof}
    $\Spec{R}$ is strictly toroidal, so $\Volsb{(\Spec{K})}=\sbb{\Spec{k}}$. In particular, if $X$ is a stably rational smooth and proper $K$-scheme, then $\Volsb(X)=\sbb{\Spec{k}}$.
\end{proof}

\begin{example}
    Consider a family 
    \begin{equation*}
    \pi\colon X\to \A^1=\Spec{k[t]}
    \end{equation*}
    such that $\X=X\times_{\A^1} \Spec{R}$ is strictly toroidal. If $\Volsb(\X_K)\neq \sbb{\Spec{k}}$ then the very general fiber $\pi_t$ over a closed point $t\in \A^1$ is not stably rational by \cref{veryGeneralCor}.
\end{example}

\begin{example} \label{variationExample}
Let $f\colon X\to S$ be a proper smooth morphism and assume that some geometric fiber is not stably rational. To show that $f$ has variation of stable birational types, it suffices by \cref{veryGeneralCor} to find two fibers that are not stably birational. Since $\Volsb$ can be computed on strictly toroidal degenerations, we have even more flexibility.

Suppose $\X \to \Spec{R}$ is strictly toroidal such that a geometric generic fiber is isomorphic to a geometric fiber of $f$. If the fibers of $f$ have constant stable birational type, then $\Volsb(\X_K)$ is the stable birational class of the fibers of $f$ by \cref{veryGeneralCor}. On the contrary, if $\Volsb(\X_K)=\sbb{\Spec{k}}$, then the stable birational type of the fibers of $f$ vary. Specifically, this means that any set of stably birational fibers is a countable union of Zariski closed sets of $S$, so $\pi$ has a variation of stable birational type.
\end{example}

\section{Stable rationality of polytopes} \label{section2}
To a Laurent polynomial $f\in k[x_1^{\pm 1},\dots, x_n^{\pm 1}]$ one can associate a convex lattice polytope $\Delta_f$, its Newton polytope. It provides a natural compactification $Z(f)$ of the zero set $\Zo(f)\subset\mathbb{G}_m^n$ in the projective toric variety induced by the polytope, and certain subdivisions of $\Delta_f$ induce degenerations of $Z(f)$. Nicaise--Ottem computes the motivic volume of these degenerations in terms of the subdivision of $\Delta_f$ \cite{NO20a}. This allows them to consider geometrically complex degenerations, and they use this to settle the question of stable rationality for many previously unknown classes of varieties. In this section, we concentrate on the combinatorial setup leading to the formula for the motivic volume. We also give an overview of the relevant invariants of $Z(f)$ that can be read from $\Delta_f$.

Some polytopes admit no nontrivial subdivisions, so the framework in \cite{NO20a} cannot be applied. At the end of this section, we collect some examples of these polytopes' possible geometric features from the viewpoint of rationality questions.

\subsection{Projective toric varieties}

Let $M\simeq \Z^n$ be a free abelian group of rank $n$ with the corresponding vector space $M_{\R} = M\otimes_\Z \R \simeq \R^n$. Let $N=\Hom_\Z(M,\Z)$ and $N_\R=N\otimes_\Z \R$ denote the dual lattice and the dual vector space with the natural pairing $\langle\cdot,\cdot\rangle \colon M_\R\times N_\R \to \R$. For any full dimensional lattice polytope $\Delta\subset M_\R$, an associated projective toric variety is constructed as follows.

Let $\sigma\subset M_\R\times \R_{\geq 0}$ be the cone defined by
\begin{equation*}
  \sigma = \settb{(\lambda v,\lambda)\in M_\R\times \R_{\geq 0}}{v\in \Delta, \lambda\in \R_{\geq 0}}.
\end{equation*}
Let $F[(M\times\Z) \cap \sigma]$ be the monoid algebra and view this as a graded algebra according to $\deg \chi^{(m,r)}=r$. The projective toric variety corresponding to $\Delta$ is defined as $\P_F(\Delta)=\Proj{F[(M\times \Z)\cap \sigma]}$. Torus-invariant sections of the line bundle $\Lb(\Delta)\defeq \OO_{\P_F(\Delta)}(1)$ are in bijective correspondence with lattice points $\Delta\cap M$. 

Let $\Sigma^P$ denote the \emph{inward normal fan} of $P$. It is the collection of cones $\sigma^Q$ in $N_\R$ for every face $Q$ of $\Delta$ defined by
 \begin{equation*}
   \sigma^Q = \sett{n\in N_\R}{\ord_\Delta(n)=\langle m,n \rangle ~\text{for all}~m\in Q}
 \end{equation*}
 where
 \begin{equation*}
   \ord_\Delta(n) \defeq \min_{m\in \Delta} \langle m,n\rangle.
 \end{equation*}
 Let $S_Q = N\cap \sigma^Q$ denote the monoid of lattice points and let $\Sigma^\Delta(r)$ denote the cones of dimension $r$.

For any ray $\rho\in \Sigma^\Delta(1)$ we let $D_\rho$ denote the torus invariant divisor corresponding to $\rho$, and $u_\rho$ the primitive generator of $\rho$.  To a Weil divisor $D=\sum_{\rho\in \Sigma^\Delta(1)} a_\rho D_\rho$ there is a polyhedron $P_D$ defined as
\begin{equation*}
P_D = \settb{m\in M_\R}{\langle m,u_\rho \rangle + a_\rho \geq 0,~ \rho\in\Sigma^\Delta(1)}.
\end{equation*}
The global torus-invariant sections of $\OO_{\P(\Delta)}(D)$ are in bijective correspondence with $P_D\cap M$ and span the vector space $H^0(\P(\Delta),\OO_{\P(\Delta)}(D))$.

Consider a Laurent polynomial $f = \sum_{m\in M} c_m\chi^m \in F[M]$ and denote by $\Supp(f)$ the set of $m\in M$ such that $c_m\neq 0$, called the \emph{support} of $f$. Let $\Delta_f$ denote the convex hull of $\Supp(f)$, the \emph{Newton polytope} of $f$.
\begin{equation*}
  \Delta_f = \Conv\settb{m\in M}{c_m\neq 0} \subset M_\R
\end{equation*}
We denote by $\Zo(f)$ the zero-locus of $f$ in the dense torus $\mathbb{G}_{m,F}^n=\Spec{F[M]}$, and by $Z(f)$ the effective Cartier divisor on $\P(\Delta)$ induced by the section of $\OO_{\P_F(\Delta)}(1)$ corresponding to $f$. 

If $\Delta$ is not full dimensional, let $m\in M\cap \Delta$ be a lattice point and consider the translation $\Delta-m= \settb{v-m\in M}{v\in \Delta}$. This is full dimensional in its linear span, and we can carry out the above construction. Then $Z(f)$ is independent of $m$: since the linear span of any such translation is the same, two different translations differ by a unimodular affine transformation in this span.

Let $f=\sum_{m\in M} c_m\chi^m$ be a Laurent polynomial with Newton polytope $\Delta$ and let $\delta\subset \Delta$ be a (lattice) subpolytope. Denote by $f_\delta=\sum_{m\in M\cap \delta} c_m\chi^m$ the restriction of $f$ to $\delta$. Then the Newton polytope of $f_\delta$ is contained in $\delta$ with equality if and only if $c_m\neq 0$ for $m$ a vertex of $\delta$. Let $\Zo(f_\delta)$ and $Z(f_\delta)$ denote the hypersurfaces obtained by the construction above.

\begin{definition} \label{newtonNonDegenDef}
  A Laurent polynomial $f$ with Newton polytope $\Delta$ is \emph{Newton nondegenerate} if $\Zo(f_\delta)$ is smooth (over $F$) for every face $\delta\subset \Delta$.
\end{definition}

This condition ensures that the singularities of $Z(f)$ can be understood combinatorially; for example, a toric resolution of singularities of $\P(\Delta)$ also produces a resolution for $Z(f)$ \cite[Prop. 3.3c]{NO20a}.
\begin{remark}
    If $f$ has Newton polytope $\Delta$, then the schematic closure of $\Zo(f)$ in $\P(\Delta)$ is $Z(f)$ \cite[Prop. 3.3a]{NO20a}. In particular, the two are birational. 
\end{remark} 
Fix a finite subset $S\subset M$ and consider the $F$-vector space $V=\Hom(S,F)$.  We say that a property holds for a general (resp. very general) Laurent polynomial with support $S$ if it holds for a non-empty Zariski open set (resp. countable intersection of Zariski open sets) in $V$. If $S=\Delta\cap M$ for a lattice polytope $\Delta$, we say that a property holds for (very) general Laurent polynomials with Newton polytope $\Delta$. By Bertini's theorem, a general Laurent polynomial with a fixed Newton polytope is Newton nondegenerate when the field has characteristic $0$.

\begin{definition} \label{rationalPolytopeDef}
  We say that a lattice polytope $\Delta$ is \emph{(stably) rational} if, for a general Laurent polynomial $f$ with Newton polytope $\Delta$, the variety $\Zo(f)$ is (stably) rational. More generally, we say that two lattice polytopes $\Delta$ and $\Delta^\prime$ are \emph{(stably) birational} if for a very general pair, $(f,g)$ of Laurent polynomials with Newton polytopes $\Delta$ and $\Delta^\prime$, respectively, $\Zo(f)$ and $\Zo(g)$ are (stably) birational.
\end{definition}

It follows from \cref{veryGeneralCor} and \cite[Prop. 3.3]{NO20a} that if there is a single pair $(f,g)$ of Newton nondegenerate Laurent polynomials with Newton polytopes $\Delta$ and $\Delta^\prime$, respectively, such that $\Zo(f)$ and $\Zo(g)$ are not stably birational. Then $\Delta$ and $\Delta^\prime$ are not stably birational. In particular, if $f$ is Newton nondegenerate and $\Zo(f)$ is not stably rational, then $\Delta_f$ is not stably rational.

A very coarse measure of the rationality of polytopes is the \emph{lattice width}. Let $\Delta$ be a (full-dimensional) lattice polytope, and let $0\neq l\in M^\vee$. The positive integer $\width(\Delta,l)=\max_{x\in \Delta}l(x) - \min_{y\in\Delta}l(y)$ is the \emph{width of $\Delta$ with respect to $l$}. The width of $\Delta$ is defined to be the minimum

\begin{equation*}
  \width(\Delta)= \min_{0\neq l\in M^\vee }\width(\Delta,l)\in \Z_{>0}.
\end{equation*}

\begin{example}[Rationality and width]
Let $\Delta\subset M_\R$ be a full dimensional lattice polytope.
If $\Delta$ has lattice width $1$, then we may assume, after a unimodular coordinate change, that $\Delta$ lies between the hyperplanes $e_1^\vee=0$ and $e_1^\vee=1$ (where $e_i$ denotes a basis with dual basis $e_i^\vee$). This means that a general Laurent polynomial with Newton polytope $\Delta$ is linear in a variable, hence rational. Already for width two polytopes, there are rational and non-stably rational examples (see \cref{TotaroKollarPolytopes}).
\end{example}

\begin{example}[Polytopes with interior lattice points]
    Let $\Delta$ be a lattice polytope and $f$ a Newton nondegenerate Laurent polynomial with Newton polytope $\Delta$. Let $X$ denote a smooth projective compactification of $\Zo(f)$. Suppose that the interior of $\Delta$ contains a lattice point. Then $X$ has nonnegative Kodaira dimension; in particular, it is not stably rational. This is a specific case of a more general construction called the \emph{Fine interior} \cite{Bat17, Rei87, Fin83}.
\end{example}

\subsection{The Fine Interior}
Let $f$ be a general Laurent polynomial with Newton polytope $\Delta$. When $\Delta$ is smooth, the asymptotic behavior of the plurigenera $P_r$ of $Z(f)$ is governed by the asymptotic behavior of the polytope $P_{K_{\P(\Delta)}+D_\Delta}$. In general, this is no longer true. However, one can still define a polytope $\Delta^{\FI}$ called the \emph{Fine interior} of $\Delta$ that captures the plurigenera of a smooth projective compactification of $\Zo(f)$. It considers the combinatorics behind a toric resolution of singularities to capture canonical sections of the resolution of $Z(f)$. J. Fine introduced it in \cite{Fin83}. For other references, see \cite{Rei87, Bat17, Bat20} for other references.

Consider the halfspaces
 \begin{align*}
   \Gamma_0^\Delta(n)&=\sett{m\in M_\R}{\ip{m,n}\geq \ord_\Delta(n)} \\
     \Gamma_1^\Delta(n)&=\sett{m\in M_\R}{\ip{m,n}\geq \ord_\Delta(n)+1}.
 \end{align*}
 The halfspaces $\Gamma_0^\Delta$ are the supporting halfspaces of $\Delta$. The \emph{Fine interior} of $\Delta$ is defined as
 \begin{equation*}
   \Delta^{\FI} = \bigcap_{0\neq n\in N} \Gamma_1^\Delta(n).
 \end{equation*}
 It suffices to take the above intersection over minimal generators of the monoids $S_Q$ for faces $Q\subset \Delta$. Equivalently, take a generating set for each cone and take the intersection of the corresponding halfspaces. Indeed if $n_1,n_2\in S_Q$ satisfying $\ip{n_i,m}\geq \ord_\Delta(n_i)+1$ then
\begin{align*}
\ip{n_1+n_2,m}=\ip{n_1,m}+\ip{n_2,m}&\geq \ord_\Delta(n_2) + \ord_\Delta(n_2)+2 \\
&\geq \ord_\Delta(n_1+n_2)+1.
\end{align*}
 This shows that $\Delta^{\FI}$ is a polytope. In general, it is not a lattice polytope; see \cref{fineInt3DimEx}.
 \begin{remark} \label{fineint_remarks}
    When $\P(\Delta)$ is smooth, we have equality $\Delta^{\FI}=P_{K_{\P(\Delta)}+D_\Delta}$. In general, $\Delta^{\FI}$ considers the supporting halfspaces one must add in a toric resolution of singularities. 
 \end{remark}
The following proposition gives a geometric interpretation of $\Delta^{\FI}$.
\begin{proposition} \label{kodaira_fine_interior}
Let $f$ be a Newton nondegenerate Laurent polynomial with Newton polytope $\Delta$, and let $X$ be a smooth projective compactification of $\Zo(f)$. Then
\begin{equation*}
  \kappa(X)=\begin{cases}
    -\infty &  \Delta^{\FI}=\emptyset \\
     \dim{\Delta^{\FI}} & 0\leq \dim{\Delta^{\FI}}<\dim\Delta \\
     \dim{\Delta^{\FI}}-1 & \dim{\Delta^{\FI}}=\dim\Delta
    \end{cases}
\end{equation*}
Moreover, if $g$ is a Newton nondegenerate Laurent polynomial with Newton polytope $\delta\subset \Delta$ and $Y$ is a smooth projective compactification of $\Zo(g)$, then $\kappa(X)\geq \kappa(Y)$.
\end{proposition}
\begin{proof}
When $\Delta^{\FI}\neq \emptyset$ this follows from \cite[Theorem 9.2]{Bat20}. When it is empty, it suffices to take a toric resolution of singularities $\widetilde{\P(\Delta)}$ of $\P(\Delta)$ and $X=Z(f)$. The canonical divisor of $X$ is the restriction of $K_{\widetilde{\P(\Delta)}}+X$ by adjunction, which has no global sections since $\Delta^{\FI}=\emptyset$. The restriction map is surjective, so the claim follows.

The property $\kappa(X)\geq \kappa(Y)$ follows from the fact that $\delta^{\FI}\subset \Delta^{\FI}$ whenever $\delta\subset \Delta$.
\end{proof}

\begin{remark}
Let $\Delta^\circ$ denote the relative interior of $\Delta$. Then $\Delta^\circ\cap M\subset \Delta^{\FI}$. Therefore, any interior lattice point is contained in the Fine interior.
\end{remark}

\begin{definition}
For a lattice polytope $\Delta$, define its Kodaira dimension $\kappa(\Delta)$ to be
\begin{equation*}
  \kappa(\Delta)=\begin{cases}
    -\infty &  \Delta^{\FI}=\emptyset \\
     \dim{\Delta^{\FI}} & 0\leq \dim{\Delta^{\FI}}<\dim\Delta \\
     \dim{\Delta^{\FI}}-1 & \dim{\Delta^{\FI}}=\dim\Delta
    \end{cases}
\end{equation*}
If $\kappa(\Delta)=\dim \Delta-1$ we say that $\Delta$ is of general type.
\end{definition}
\begin{example} \label{fineInt3DimEx}
The polytope $\Delta$ given by
\begin{equation*}
\Delta = \Conv \set{2e_2+2e_3, e_1+3e_2,2e_1+4e_2+3e_3, 3e_1+e_3}
\end{equation*}
has no interior lattice points (i.e., it is \emph{hollow}). The Fine interior $\Delta^{\FI}$ is the following 3-dimensional polytope\footnote{A Macaulay2 package containing methods to compute Fine interior and lattice width can be found at \url{https://github.com/simen94/M2/}.}.
\begin{align*}
\Delta^{\FI} = \frac{1}{5}\Conv \{7e_1&+12e_2+6e_3, 9e_1+9e_2+7e_3, \\ &6e_1+11e_2+8e_3,8e_1+13e_2+9e_3\}.
\end{align*}
So $\kappa(\Delta)=3$. There is only a finite number of hollow 3-dimensional polytopes (up to unimodular equivalence) with non-empty Fine interior \cite{BKS19}.
\end{example}

\subsection{Subdivisions}
Let $\Delta\subset M_\R$ be a lattice polytope. We say that a set $\PP$ of subpolytopes of $\Delta$ is a \emph{polyhedral subdivision} if
\begin{itemize}
\item[1)] $\cup_{\alpha\in \PP} \alpha = \Delta$,
\item[2)] for $\alpha,\beta\in \PP$ we have $\delta=\alpha\cap \beta\in \PP$ and $\delta$ is a face of both $\alpha$ and $\beta$.
\end{itemize}
 We say that elements of $\PP$ are cells of the subdivision. Furthermore, we say that a polyhedral subdivision $\PP$ is
\begin{enumerate}
    \item[1)] \emph{integral} if all polytopes in $\PP$ are lattice polytopes,
    \item[2)] \emph{regular} if there is a convex piecewise linear function $\phi\colon \Delta\to \R$ such that the linear domains of $\phi$ are exactly the cells of $\PP$.
\end{enumerate}

\begin{figure}[!h]
\begin{subfigure}{.45\textwidth}
\centering
\tdplotsetmaincoords{60}{-20}
\begin{tikzpicture}[tdplot_main_coords]	
\foreach \x in {0,...,2}{
    \tikzmath{\xx = 2 - \x;}
    \foreach \y in {0,...,\xx}{
        \draw[fill=black] (\x,\y) circle (1pt);
}}
\draw[fill=blue, opacity=0.3] (2,0,2) -- (1,0,1) -- (1,1,1) -- cycle;
\draw[fill=red, opacity=0.3] (0,0,1) -- (1,0,1) -- (1,1,1) -- (0,1,1) -- cycle;
\draw[fill=green, opacity=0.3] (0,1,1) -- (1,1,1) -- (0,2,2) -- cycle;
\draw[fill=blue, opacity=0.5] (2,0,0) -- (1,0,0) -- (1,1,0) -- cycle;
\draw[fill=red, opacity=0.5] (0,0,0) -- (1,0,0) -- (1,1,0) -- (0,1,0) -- cycle;
\draw[fill=green, opacity=0.5] (0,1,0) -- (1,1,0) -- (0,2,0) -- cycle;
\draw[dotted] (0,2,0) -- (0,2,2);
\draw[dotted] (0,1,0) -- (0,1,1);
\draw[dotted] (0,0,0) -- (0,0,1);
\draw[dotted] (1,1,0) -- (1,1,1);
\draw[dotted] (2,0,0) -- (2,0,2);
\draw[dotted] (1,0,0) -- (1,0,1);
\end{tikzpicture}
\caption{Integral regular polyhedral subdivision}
\label{fig:nonint}
\end{subfigure}
\begin{subfigure}{.45\textwidth}
\centering
\begin{tikzpicture}[scale=0.8]
    \draw (0,0)--(3,0)--(1.5,3)--(0,0);
    \draw (1,0.8)--(2,0.8)--(1.5,1.8)--(1,0.8);
    \draw (0,0)--(1,0.8);
    \draw (0,0)--(2,0.8);
    \draw (3,0)--(2,0.8);
    \draw (3,0)--(1.5,1.8);
    \draw (1.5,3)--(1.5,1.8);
    \draw (1.5,3)--(1,0.8);
    \foreach \Point in {(0,0), (3,0), (1.5,3), (1,0.8), (1.5,1.8), (2,0.8)}{
       \node at \Point {\textbullet};
    }
\end{tikzpicture}
\caption{Non-regular subdivision \cite[Example 2.2.5]{DRS10}}
\label{fig:nonreg}
\end{subfigure}
\end{figure}

\begin{example} \label{pullingSubDivEx}
  Let $\Delta$ be a lattice polytope and $\PP$ a regular subdivision. Let $p\in \Delta\cap M$ be a lattice point. The \emph{pulling refinement} of $\PP$ by the point $p$ is a subdivision $\operatorname{pull}_p(\PP)$ consisting of the following faces.
  \begin{itemize}
      \item[1)] If $\alpha\in \PP$ and $p\not\in \alpha$ then $\alpha\in \operatorname{pull}_p(\PP)$.
      \item[2)] If $\alpha\in \PP$ and $p\in\alpha$ then $\Conv(p,\beta)\in\operatorname{pull}_p(\PP)$ for every face $\beta$ of $\alpha$ with $p\not\in \beta$. 
  \end{itemize}
If $\PP$ is regular then $\operatorname{pull}_p(\PP)$ is also regular \cite[Lemma 4.3.12]{DRS10}. Indeed if a convex piecewise linear function induces $\PP$, then pulling down the value at $p$, one obtains a convex function inducing $\operatorname{pull}_p(\PP)$.
\end{example}

From a regular integral polyhedral subdivision $\PP$ of a polytope $\Delta$, one can construct a strictly toroidal degeneration \cite[Section 3.6]{NO20a}. Suppose $f$ is Newton nondegenerate with Newton polytope $\Delta$. From $\PP$, one can construct a degeneration such the special fiber consists of strata of the form $Z(f_\delta)$ for $\delta\in \PP$ but not contained in the boundary of $\Delta$. This leads to the following theorem, which is the primary tool for studying the stable rationality of polytopes.

\begin{theorem}[{\cite[Theorem 3.15]{NO20a}}] \label{tropicalVolume}
    Let $k$ be a field of characteristic 0. Consider $\R^{n+1}$ with its standard lattice $\Z^{n+1}$ for some positive integer $n$. Let $\Delta$ be a lattice polytope of dimension at least 2, and let $\PP$ be a regular integral polyhedral subdivision of $\Delta$. Let $f\in k[\Z^{n+1}]$ be a Laurent polynomial with Newton polytope $\Delta$ and assume that for every face $\delta\in \PP$ the variety $\Zo(f_\delta)$ is smooth over $k$.

    Suppose that 
    \begin{equation} \label{volumeFormula}
        (-1)^{n+1}\sum_{\substack{\alpha\in\PP \\ \alpha\centernot\subset \partial\Delta}} (-1)^{\dim{\alpha}}\sbb{\Zo(f_\alpha)} \neq \sbb{\Spec{k}}
    \end{equation}
    in $\SBB{k}$, then $\Delta$ is not stably rational.
\end{theorem}

Most degenerations (or subdivisions) do not lead to an obstruction in the sense of \cref{tropicalVolume}. But they are still helpful for studying variation of stable birational type.

\begin{definition}
    Let $\X$ be a strictly toroidal $R$-scheme. We say that $\X$ is an \emph{unobstructed} degeneration if $\Volsb(\X_K)=\sbb{\Spec{k}}$. Similarly, we say that a subdivision is \emph{unobstructed} if it induces an unobstructed degeneration.
\end{definition}

\begin{example}
    Let $\Delta$ be a lattice polytope with $\kappa(\Delta)=-\infty$. Then any unimodular triangulation of $\Delta$ is unobstructed.
\end{example}

\subsection{Empty and Relatively Empty Polytopes}
Subdivisions, and more generally, strictly toroidal degenerations, are the primary tool for studying stable rationality of polytopes. From this point of view, the polytopes that admit no subdivisions are atomic pieces. Two types are fundamentally important: The \emph{empty simplices}, admitting no nontrivial polyhedral subdivisions, and \emph{relatively empty polytopes}, where all but one lattice point is concentrated in a single face.

\begin{definition}
    Let $\Delta\subset M_\R$ be a lattice polytope. We say that $\Delta$ is
    \begin{itemize}
      \item \emph{empty} if the set of lattice points equals the set of vertices. We say that $\Delta$ is an \emph{empty simplex} if this set has size $\dim(\Delta)+1$,
      \item \emph{relatively empty} if there is a face $\delta\subset\Delta$ such that
      \begin{align*}
          |M\cap (\Delta\setminus \delta)| = 1.
      \end{align*}
      In this case, $\Delta$ is empty relative to $\delta$.
      \item \emph{hollow} if $\Delta^\circ\cap M=\emptyset$.
    \end{itemize}
    In particular, empty $\implies$ relatively empty $\implies$ hollow.
\end{definition}
\begin{remark}
If $\Delta$ is a relatively empty polytope relative to $\delta$, then any nontrivial subdivision restricts to a nontrivial subdivision in $\delta$.
\end{remark}
    
\begin{figure}[!ht]
  \begin{subfigure}{.4\textwidth}
      \centering
      \tdplotsetmaincoords{80}{410}
\begin{tikzpicture}[tdplot_main_coords, scale=0.6]	

\draw (0,0,0) -- (4,0,0);
\draw[dashed] (0,0,0) -- (0,4,0);
\draw (0,0,0) -- (-3,0,4);
\draw (4,0,0) -- (0,4,0);
\draw (4,0,0) -- (-3,0,4);
\draw (0,4,0) -- (-3,0,4);

\draw[fill=gray] (0,0,0) circle (3pt);
\draw[fill=gray] (4,0,0) circle (3pt);
\draw[fill=gray] (0,4,0) circle (3pt);
\draw[fill=gray] (-3,0,4) circle (3pt);

\end{tikzpicture}
      \caption{Empty simplex}
      \label{fig:emptySimp}
  \end{subfigure}
  \hspace{0.5cm}
  \begin{subfigure}{.4\textwidth}
      \centering
      \tdplotsetmaincoords{80}{410}
\begin{tikzpicture}[tdplot_main_coords, scale=0.6]	
\draw (0,0,0) -- (4,0,0);
\draw[dashed] (0,0,0) -- (0,4,0);
\draw (0,0,0) -- (-2,-1,4);
\draw (4,0,0) -- (0,4,0);
\draw (4,0,0) -- (-2,-1,4);
\draw (0,4,0) -- (-2,-1,4);
\draw[fill=gray] (0,0,0) circle (3pt);
\draw[fill=gray] (4,0,0) circle (3pt);
\draw[fill=gray] (0,4,0) circle (3pt);
\draw[fill=gray] (-2,-1,4) circle (3pt);
\draw[fill=gray] (2,2,0) circle (3pt);
\draw[fill=gray] (2,0,0) circle (3pt);
\draw[fill=gray] (0,2,0) circle (3pt);
\end{tikzpicture}
      \caption{A relatively empty simplex}
      \label{fig:relemptySimp}
  \end{subfigure}
\end{figure}

The following proposition shows empty simplices have a unique isomorphism class in the linear system.

\begin{proposition} \label{birationalTypeEmptySimp}
    Let $F$ be an algebraically closed field, and $\Delta\subset \R^n$ an empty simplex. Let $f,g\in F[M]$ be Laurent polynomials with Newton polytope $\Delta$. Then $Z(f)$ and $Z(g)$ are isomorphic.
\end{proposition}
\begin{proof}
The linear system $V = \P H^0(\P(\Delta),\Lb(\Delta))$ is $n$-dimensional with coordinates in correspondence with the vertices of $\Delta$. There is an induced torus action of $(\C^\ast)^n$ on $V$, and for dimension reasons, the orbit of any element of $V$ under this action is $V$. So there is a unique section of $V$ (with Newton polytope $\Delta$) up to scaling by $(\C^\ast)^n$.
\end{proof}

\begin{example} \label{generalTypePolytope}
    Empty simplices can have large Fine interiors. Consider the  polytope $\Delta\subset \R^4$ given as
    \begin{equation*}
       \Delta = \Conv \{0, e_1, e_2, e_3, e_4, 6e_1+14e_2+17e_3+65e_4 \}.
  \end{equation*}
  This is an empty simplex, and the Fine interior $\Delta^{\FI}$ is full dimensional, so $\Delta$ is of general type.
\end{example}

\begin{example} \label{TotaroKollarPolytopes}
    In contrast to \cref{generalTypePolytope}, there exist empty simplices with an empty Fine interior, which is not stably rational. Consider $\Delta_{n,d}\subset \R^{n+1}$ given as the convex hull following.
    \begin{align*}
    \Delta_{n,d} = \Conv \set{2e_1, e_2, (d-1)e_2+e_3, \dots, (d-1)e_{n} + e_{n+1}, (d-1)e_{n+1}}
    \end{align*}
    This is the Newton polytope of a double cover of $\P^n$ branched along the hypersurface defined by $x_0^{d-1}x_1+\dots +x_{n-1}^{d-1}x_n + x_n^{d-1}x_0$. \Kollar proved that these are not rational when $d\geq 2\ceil{(n+3)/3}$ and $n,d$ are even \cite{Kol95}. Totaro improved this and showed that they are not stably rational when $d\geq 2\ceil{(n+2)/3}$ \cite{Tot16}. The polytope $\Delta_{4,4}$ gives an example of a non-stably rational empty simplex with an empty Fine interior.
\end{example}

\begin{example} \label{rationalCubicsEmpty}
The following polytopes are examples of rational empty simplices birational to cubic hypersurfaces. For $n$ odd, consider the cubic hypersurfaces defined by $x_0^2x_1+\dots + x_{n-1}^2x_n + x_n^2x_0$. These contain disjoint linear subspaces of dimension $(n-1)/2$, so they are rational. The Newton polytopes are given by
\begin{equation*}
\Conv \{e_1, 3e_1+e_2, \dots, 3e_{n-1} + e_{n}, 3e_{n}\}
\end{equation*}
and they are empty simplices.
\end{example}

\begin{remark}
We have the following low-dimensional classification of empty simplices.
\begin{itemize}
    \item[] $\dim =1$. A unique one, a line segment of length 1.
    \item[] $\dim =2$. The polytope is unimodularly equivalent to the simplex $\Delta_2$.
    \item[] $\dim =3$. The polytope is unimodular equivalent to a tetrahedron
    \begin{equation*}
        T(p,q) = \Conv{\set{0, e_1, e_3, pe_1+qe_2+e_3}}
    \end{equation*}
    where $p\in \set{1,\dots, q}$ and $\gcd(p,q)=1$ \cite{Wh64}. All of these are of width 1.
    \item[] $\dim =4$. A complete classification is given in \cite{IS21}. An infinite number of these are of widths 1 and 2. Finitely many of width 3 and a unique simplex of width 4 (the one in \cref{generalTypePolytope}).
    \item[] $\dim \geq 5$. No classification is known.
\end{itemize}
\end{remark}

\begin{example}
Since the stable birational type in empty simplices is constant, one cannot control stable birational types using variation. So, the following question is key to dealing with these.
\begin{itemize}
\item[(Q1)] Are there examples of empty simplices $\Delta,\Delta'$ that are not unimodular equivalent, not stably rational, but stably birational?
\end{itemize}
\end{example}

\begin{example}
Let us give a geometric interpretation of the problem caused by relatively empty polytopes. Suppose $\delta$ is not stably rational and is contained in a larger polytope $\Delta'$. Our strategy to prove that $\Delta'$ is not stably rational is to construct a subdivision containing $\delta$ and show that the obstruction in \cref{tropicalVolume} is nontrivial. Since $\delta$ and polytopes that have it as a face may appear with opposite signs, after possibly subdividing further, we end up needing the following property to ensure that no cancelation occurs.

Let $f$ be a very general Laurent polynomial with Newton polytope $\Delta$ such that $\Delta$ contains $\delta$ as a face and is empty relative to $\delta$. Then $Z(f)$ is birational to a cyclic cover of projective space with a component in the branch locus birational to $Z(f_\delta)$. This leads to the following natural question, to which we do not know the answer.
\begin{itemize}
\item[(Q2)] Is there a cyclic cover $X\to \P^n$ branched along $Z\subset \P^n$ such that $X$ is not stably rational but stably birational to a component of $Z$? 
\end{itemize}
\end{example}

\section{Variation of stable birational type} \label{section3}
Given a convex polytope $\Delta$, one asks whether it is stably rational. The strategy developed in \cite{NO20a} is as follows. Identify a subpolytope $\delta\subset\Delta$ that is not contained in the boundary of $\Delta$. Then construct a subdivision of $\Delta$ such that the obstruction in \cref{tropicalVolume} is nontrivial. Typically, one will try to construct a subdivision containing $\delta$ such that all other cells are stably rational. This will prevent cancellation in \eqref{volumeFormula}. 

It is unclear whether we can arrange a situation where no term in \eqref{volumeFormula} cancels with the term corresponding to $\delta$. It is a problem concerning a strong type of variation of stable birational type; we need to vary the stable birational type while keeping monomials contained in a face constant. This section shows that this property holds for a large class of polytopes. This allows us to prove that the property of not being stably rational is preserved under inclusion for these polytopes. 
\begin{definition} \label{variationPolytopes}
    Let $\alpha,\beta\subset M_\R$ be lattice polytopes. We say that
    \begin{itemize}
    \item[1)] $\alpha$ has \emph{variation of stable birational type} (variation) if $\sbb{\Zo(f)}$ is non-constant as $f$ varies over Newton non-degenerate Laurent polynomials with Newton polytope $\Delta$.
    \item[2)] The pair $(\alpha,\beta)$ has \emph{variation of stable birational type} (variation) if
    \begin{equation*}
    \sbb{\Zo(f_\delta)} \neq \sbb{\Zo(f_\alpha)}.
    \end{equation*}
    For a very general Laurent polynomial $f$ with support $\alpha\cup\beta$. 
    \end{itemize}
    Fix a lattice polytope $\delta\subset M_\R$. Then
    \begin{itemize}
    \item[3)] $\delta$ has \emph{strong variation of stable birational type} (strong variation)  if $(\delta,\alpha)$ has variation for any $\alpha\subset M_\R$ with $\alpha\cap \delta^\circ=\emptyset$.
    \end{itemize}
    In particular, when $\alpha\cap\beta=\emptyset$, then $(\alpha,\beta)$ has variation if and only if $\alpha$ is not stably birational to $\beta$.
\end{definition}

\begin{remark}
Note that the third property implies the first by taking $\alpha$ as a translation of $\delta$ such that $\alpha\cap\delta=\emptyset$.
\end{remark}

\begin{example}
The empty simplices do not have variation by \cref{birationalTypeEmptySimp}. Similarly, relatively empty polytopes do not have strong variation, but they may have variation.
\end{example}

\begin{example}
The following cases are established in \cite[Theorem 4.1]{NO20a}.
\begin{itemize}
    \item If $\alpha$ has an unobstructed subdivision and is not stably rational, then it has variation.
    \item $(\alpha,\beta)$ has variation when $\alpha\cap \beta = \emptyset$ and when $\alpha$ or $\beta$ has variation.
    \item $(\delta,\Delta)$ has variation when $\delta\subset \Delta$ is a face and $\Delta$ admits a subdivision $\PP$ such that $\delta\in\PP$ and every face $\alpha\in \PP\setminus \set{\delta}$ with $\alpha\cap \delta\neq \emptyset$ is stably rational.
\end{itemize}
\end{example}

To study the variational properties of polytopes, it is useful to have a generalization of \cite[Prop. 3.10]{NO20a}. It states that a polytope $\Delta$ is not stably rational if there is a Newton non-degenerate Laurent polynomial $g\in F[M]$ with Newton polytope $\Delta$ for some algebraically closed field $F$ of characteristic 0 such that $\Zo(g)$ is not stably rational. A similar property holds for pairs of polytopes using more or less the same argument as in \cite[Prop. 3.10]{NO20a}.

\begin{proposition} \label{fieldReductionProp}
Let $\alpha_i\subset M_\R,~i=1,2$ be lattice polytopes. If for some algebraically closed field $F$ of characteristic 0, there is a Newton non-degenerate Laurent polynomial $g\in F[M]$ such that

  \begin{enumerate}[label=\arabic*)]
  \item $g_{\alpha_i}$ has Newton polytope $\alpha_i$,
  \item $g_{\alpha_i}$ is Newton non-degenerate,
  \item $\Zo(g_{\alpha_1})$ is not stably birational to $\Zo(g_{\alpha_2})$.
  \end{enumerate}
  Then a very general Laurent polynomial $g$ over any algebraically closed field of characteristic 0 satisfying the condition (1-2) also satisfies (3). In particular, $\alpha_1$ and $\alpha_2$ are not stably birational.
\end{proposition}
\begin{proof}
  Let $\Delta$ be the convex hull of $\alpha_1\cup \alpha_2$ and $\pi\colon X\to \P_{\Q}(\Delta)$ a toric resolution of singularities. 
  Consider the universal family
  \begin{equation*}
    \theta\colon\Y\times_{\P H^0(\P_{\Q}(\Delta), \Lb(\Delta))} U\to U
  \end{equation*}
  where $U\subset \P H^0(\P_{\Q}(\Delta), \Lb(\Delta))$ is the Zariski open subset parameterizing hypersurfaces with Newton polytope $\Delta$ and $\Y$ is the universal family of the linear system $\pi^\ast\Lb(\Delta)$ (see \cite[Prop. 3.10]{NO20a} for details). For any algebraically closed field $F$ of characteristic 0, the $F$-points $u$ of $U$ are parametrizing Laurent polynomials $g\in F[M]$ with Newton polytope $\Delta$ up to scaling by a unit in $F$. Similarly, the fibers $\theta^{-1}(u)$ are realized as the closure of $\pi^{-1}(\Zo(g))$ in $X\times_{\Q}F$. Similarly, let
  \begin{align*}
    &\theta_{\alpha_1}\colon\Y_{\alpha_1}\times_{\P H^0(\P_{\Q}(\alpha_1), \Lb(\alpha_1))} U_{\alpha_1}\to U_{\alpha_1} \\
    &\theta_{\alpha_2}\colon\Y_{\alpha_2}\times_{\P H^0(\P_{\Q}(\alpha_2), \Lb(\alpha_2))} U_{\alpha_2}\to U_{\alpha_2}
  \end{align*}
  denote the corresponding universal families constructed from $\alpha_1$ and $\alpha_2$. Let $U'_{\alpha_i}\subset U$ denote the open dense subset of Laurent polynomials $g$ such that $g_{\alpha_i}$ has Newton polytope $\alpha_i$. The two maps
  \begin{equation*}
    \begin{tikzcd}
{\P H^0(\P_{\Q}(\Delta), \Lb(\Delta))} \arrow[d, "\phi_{\alpha_1}", dashed] \arrow[r, "\phi_{\alpha_2}", dashed] & {\P H^0(\P_{\Q}(\alpha_2), \Lb(\alpha_2))} \\
{\P H^0(\P_{\Q}(\alpha_1), \Lb(\alpha_1))}
\end{tikzcd}
  \end{equation*}
  are both defined on $U'=U'_{\alpha_1}\cap U'_{\alpha_2}$. The map $\phi_{\alpha_i}$ maps $U'$ into $U_{\alpha_i}$. By Bertini's Theorem, there is a dense open subset $V\subset U'$ such that the following maps are smooth.
  \begin{align*}
     \psi_{\alpha_1}\colon\Y_{\alpha_1}\times_{\P H^0(\P_{\Q}(\alpha_1), \Lb(\alpha_1))} V\to V, \qquad \psi_{\alpha_2}\colon\Y_{\alpha_2}\times_{\P H^0(\P_{\Q}(\alpha_2),\Lb(\alpha_2))} V\to V.
  \end{align*}
  Let $g$ denote an $F$-point of $V$. Then the fiber $\psi_{\alpha_i}^{-1}(g)$ is the closure of $\Zo(g_{\alpha_i})$ in $X_{\alpha_i}\times_\Q F$. In particular, the set of points in $V$ with stably birational fibers is a countable union of Zariski closed subsets by \cref{veryGeneralCor}. In particular, if a single pair of geometric fibers is not stably birational, then the fibers over a very general geometric point of $V$ are not stably birational.
\end{proof}

We first note that if the intersection of the polytopes in question is small enough, we have variation.

\begin{proposition}
    Let $\alpha,\beta\subset M_\R$ be non-stably irrational lattice polytopes that meet on a common face $\alpha\cap\beta=\delta$. If $\delta$ is an empty simplex (e.g., a point) and either $\alpha$ or $\beta$ has variation of stable birational type (e.g., admits an unobstructed subdivision), then $(\alpha,\beta)$ has variation of stable birational type.
\end{proposition}
\begin{proof}
    Let $g$ be a very general Laurent polynomial with Newton polytope $\Delta$. By \cref{birationalTypeEmptySimp}, we may assume that $g_\delta$ is constant; hence, the question is equivalent to asking whether $\sbb{\Zo(f)}\neq \sbb{\Zo(g)}$ where $f$ and $g$ are very general with Newton polytope $\alpha$ and $\beta$, respectively. The result now follows since $\alpha$ or $\beta$ has variation of stable birational type.
\end{proof}

\subsection{Smooth subpolytopes}
The goal for the rest of this section is to study variation of polytopes. Our strategy is constructing unobstructed degenerations (cf. \cref{variationExample}). To prove strong variation, we need to construct sufficiently many unobstructed degenerations; specifically, this means that for any face, we construct an unobstructed degeneration that is trivial when restricted to the monomials supported in this face. We first show that this holds for certain smooth polytopes, then give a condition for singular polytopes.

A lattice polytope $\Delta$ is smooth if the following property holds for every vertex $v\in \Delta$.
\begin{itemize}
    \item[] After a translation, assume that $v$ is the origin of $M_\R$. Then the cone spanned by the edges containing $v$ is simple. That is, primitive generators form part of a basis for $M$.
\end{itemize}
For a lattice polytope $\Delta\subset M_\R$ and for a ray $\rho\in\Sigma^{\Delta}(1)$ denote by $F_\rho$ the facet of $\Delta$ such that $\rho$ is normal to $F_\rho$. Let $D_\Delta$ be a torus-invariant divisor such that $\OO_{\P(\Delta)}(D_\Delta)=\Lb(\Delta)$. The global torus invariant sections of $D_\Delta$ correspond bijectively to $\Delta\cap M$, and the global torus invariant sections of $D_\Delta-D_\rho$ correspond bijectively to the lattice points of the polytope obtained by taking $\Delta$ and shifting the supporting half-space containing the span of $F_\rho$ inwards by one. We denote this polytope by $\Delta(-F_\rho)$. In general, this is not a lattice polytope, but it is when $\Delta$ is smooth.

\begin{lemma}
    Let $\Delta$ be a smooth lattice polytope. Then $\Delta(-F_\rho)$ is a lattice polytope for all $\rho\in\Sigma^\Delta(1)$.
\end{lemma}
\begin{proof}
We may assume that $\Delta$ is full dimensional. Let $v$ be a vertex of $F_\rho$ and assume $v=0$ after translation. By the smoothness assumption, the set of primitive generators $B = \settb{u_E}{v\in E~, E~\text{edge}}$ forms a basis for $M$. There is a subset of $n-1$ of these that forms a basis for (the necessarily full sub lattice) $\Span{F_\rho}\cap M$. The hyperplane obtained by shifting inward by one must contain the lattice point $u_E$; if not, there are lattice points that are not contained in the span of $B$, contradicting that $B$ is a basis.
\end{proof}

\begin{lemma} \label{bpfLemma}
    If $\Delta$ is smooth, then the divisor $D_\Delta-D_\rho$ is basepoint free for all $\rho\in \Sigma^\Delta(1)$.
\end{lemma}
\begin{proof}
    Let $\Sigma^\Delta$ be the normal fan of $\Delta$ and $\Sigma^{\Delta(-F_\rho)}$ the normal fan of $\Delta(-F_\rho)$. Then $\Sigma^\Delta$ is a refinement of $\Sigma^{\Delta(-F_\rho)}$ so we have a birational morphism $\P(\Delta)\to \P(\Delta(-F_\rho))$ and by \cite[Prop. 6.2.7]{CLS11} we have
    \begin{equation*}
        \phi^\ast D_{\Delta(-F_\rho)} - E = D_\Delta - D_\rho
    \end{equation*}
for some $E$ with $\phi_\ast E = 0$. Since $P_{\phi^\ast D_{\Delta(-F_\rho)}}$ and $P_{D_\Delta-D_\rho}$ differ by a translation, we have $E\sim 0$, so $D_\Delta-D_\rho$ is the pullback of an ample divisor, hence basepoint free.
\end{proof}

Recall that the homogeneous coordinate ring $S$ of a toric variety $\P_k(\Delta)$ is the algebra
\begin{align*}
S = k[x_\rho~|~\rho\in \Sigma^\Delta(1)]
\end{align*}
graded by $\deg(x_\rho)=[D_\rho]\in \Cl(\P_k(\Delta))$ \cite{Cox95}. A homogeneous ideal in $S$ defined a closed subscheme of $\P_k(\Delta)$, in particular the ideal $(x_\rho)$ defines the closed subscheme $Z(x_\rho)=D_{\rho}$.

\begin{lemma} \label{smoothDegeneration}
    Let $X=\P_k(\Delta)$ be a smooth toric variety with homogeneous coordinate ring $S$. Let $f$ a Newton non-degenerate Laurent polynomial with Newton polytope $\Delta$ (an element of $S$ with degree $[D_\Delta]$). Let $s\in S$ be general with $\deg(s) = [D_\Delta - D_\rho]$. Then the $R$-scheme
    \begin{align*}
        \X = Z(ft-x_\rho s)\subset X\times_k R
    \end{align*}
    is strictly toroidal.
\end{lemma}
\begin{proof}
The subscheme $Z(s)$ is a general element of the linear system of $\OO_X(D_\Delta-D_\rho)$ and $Z(x_\rho)$ is the torus invariant subscheme $D_\rho$.
It follows from Bertini's theorem and \cref{bpfLemma} that the strata of $\X$ are smooth of the correct dimension. Let $U=\Spec{R[(\sigma^\vee \cap M)]}$ be a torus-invariant open set of $X\times_k R$ corresponding to a cone $\sigma$. The scheme $Z(ft-x_\rho s)\cap U$ is defined by $f't-x_\rho^\prime s'$ where $s',f'$ and $x_\rho^\prime$ are regular functions on $U$, all of them a local generator for the ideal sheaf of the subscheme defined by their global counterpart (here we are using the smoothness of $X$). We have
\begin{equation*}
\X\cap U = \Spec{R[(\sigma^\vee \cap M)]/(ft-x_\rho s)}.
\end{equation*}
and this is the subscheme cut out by $(f-\chi^{e_1},s-\chi^{e_2})$ in the strictly toroidal $R$-scheme
\begin{equation*}
\Spec{R[(\sigma^\vee \cap M)\oplus \N^2]/(\chi^{e_1} t - x_\rho \chi^{e_2})} \simeq \Spec{R[N']/(t-\chi^{e_3})}.
\end{equation*}
Here $N'$ is the monoid $((\sigma^\vee \cap M)\oplus \N^3)/\sim$ where $e_1 + e_3 \sim e_2 + e_\rho$ and $x_\rho = \chi^{e_\rho}$. The claim now follows from \cref{logSubSchTheorem}.
\end{proof}

\begin{lemma}\label{monomialLemma}
    Let $\P(\Delta)$ a toric variety, $D=\sum_\rho a_\rho D_\rho$ a Weil divisor, and $s$ a global section of $\OO_{\P(\Delta)}(D)$. Let $Z(s)$ denote the effective divisor linearly equivalent to $D$ induced by $s$. Let $H_{\rho_0}$ be a face of $P_D$ and $\rho_0$ the ray dual to it. Suppose 
    \begin{align*}
        Z(s)=\sum_{\rho\in \Sigma^\Delta(1)} b_\rho D_\rho.
    \end{align*}
    with $b_{\rho_0}\neq 0$. Then $s$ corresponds to a lattice point in $P_D \setminus H_{\rho_0}$.
\end{lemma}
\begin{proof}
    The polyhedron $P_D$ is defined by the intersection of halfspaces
    \begin{align*}
        \settb{m\in M_\R}{(m,u_\rho)+a_\rho \geq 0}
    \end{align*}
    where $u_\rho$ is the primitive generator of the ray $\rho$. Suppose $s$ corresponds to a lattice point $m_0\in H_{\rho_0}$, then $(m_0,u_{\rho_0})+a_{\rho_0}=0$. Moreover,
    \begin{align*}
    \operatorname{div}(s)=& \operatorname{div}(\chi^{m_0}) + D\\ 
    =&\sum_{\rho\in \Sigma^\Delta(1)} ((m_0,u_\rho) + a_\rho)D_\rho
    \end{align*}
    so $b_{\rho_0}=(m_0,u_{\rho_0})+a_{\rho_0}=0$ and the claim follows.
\end{proof}

\begin{proposition} \label{smoothDegenThm}
    Let $\delta$ be a smooth lattice polytope with facets $F_\rho, \rho\in \Sigma^\delta(1)$ such that
    \begin{enumerate}
        \item $\delta$ is not stably rational, and
        \item the polytopes $\delta(-F_\rho)$ admits an unobstructed subdivision for every $\rho\in \Sigma^\delta(1)$.
    \end{enumerate}
    Then $\delta$ has strong variation of stable birational type.
\end{proposition}
\begin{proof}
Let $\alpha$ be any lattice polytope such that $\alpha\cap \delta^\circ = \emptyset$ and $\Delta$ a lattice polytope containing $\alpha\cup\delta$. Furthermore, let $g$ be a very general Laurent polynomial with Newton polytope $\Delta$.

If $\alpha$ is stably rational, then we are done, so we may assume that $\alpha$ is not stably rational. Let $\OO(D_\delta)$ be the ample line bundle on $\P(\delta)$.
Let $F_{\rho}$ denote a facet of $\delta$ such that $\alpha\cap \delta \subset F_{\rho}$. By \cref{fieldReductionProp} it suffices to construct a Laurent polynomial $f$ such that $f_\alpha$ and $f_\delta$ are Newton non-degenerate with Newton polytopes $\alpha$ and $\delta$, respectively, such that $\Zo(f_\delta)$ is not stably birational to $\Zo(f_\alpha)$. 

Let $S$ denote the homogeneous coordinate ring of $\P_k(\delta)$. Let $s$ be general with $\deg(s) = [D_\delta - D_\rho]$ and $g_\delta$ be general with $\deg(g_\delta) = [D_\delta]$. 
Define
\begin{equation*}
f_\delta = g_\delta t - x_\rho s \in S\otimes_k K.
\end{equation*}
There is an isomorphism between the vector space of homogeneous polynomials in $S$ of degree $D_\Delta$ and $H^0(\P_k(\delta), \OO(D_\Delta))$ \cite[Prop 1.1]{Cox95}. Using this, we can identify $f_\delta$ with a Laurent polynomial $f'_\delta$ such that $Z(f_\delta)=Z(f'_\delta)$. We denote both by $f_\delta$ for ease of notation. We let $g$ denote a very general Laurent polynomial with Newton polytope $\Delta$ such that the restriction to $\delta$ is indeed the $g_\delta$ constructed above. Now $f=gt-x_\rho s$ is a Newton nondegenerate Laurent polynomial with Newton polytope $\Delta$. It suffices to show that $\Vol(Z(f_\alpha))\neq \Vol(Z(f_\delta))$. By \cref{monomialLemma} we have
\begin{equation*}
f_\alpha = g_\alpha t, \quad f_\delta = g_\delta t - x_\rho s
\end{equation*}
and since $g$ and $s$ are general, these (and also $f$ itself) are Newton non-degenerate with the correct Newton polytopes. By \cref{smoothDegeneration} the closure $Z(f_\delta)\subset X\times_k R$ is strictly toroidal, so
\begin{align*}
\Volsb(Z(f_\delta)) &= \sbb{\Spec(k)} + \sbb{Z(s)} - \sbb{Z(s)\cap Z(x_\rho)} \\
\Volsb(Z(f_\alpha)) &= \sbb{Z(g_\alpha)}.
\end{align*}
Thus it suffices to show 
\begin{align*}
\sbb{\Zo(g_\alpha)} \neq \sbb{\Spec(k)} + \sbb{Z(s)} - \sbb{Z(s)\cap Z(x_\rho)}.
\end{align*}
Equality is only possible if $\sbb{\Zo(g_\alpha)}=\sbb{Z(s)}$. Here $\sbb{Z(s)}=\sbb{\Zo(h)}$ where $h$ is a very general Laurent polynomial with Newton polytope $\delta(-F_\rho)$, in particular, this no longer depends on $g_\alpha$. Since $\delta(-F_\rho)$ admits an unobstructed subdivision, it follows that $\sbb{\Zo(g_\alpha)} \neq \sbb{\Zo(h)}$.
\end{proof}

\begin{theorem} \label{smoothContainment}
Let $\Delta$ be a lattice polytope and $\delta\subset \Delta$ a smooth lattice polytope 
 with $\dim\delta=\dim\Delta$ such that
\begin{enumerate}
        \item $\delta$ is not stably rational, and
        \item $\delta(-F_\rho)$ admits an unobstructed subdivision for every $\rho\in\Sigma^{\delta}(1)$.
    \end{enumerate}
    Then $\Delta$ is not stably rational.
\end{theorem}
\begin{proof}
Let $g$ be a very general Laurent polynomial with Newton polytope $\Delta$. Consider the regular subdivision $\PP$ given by the lower convex envelope of
\begin{align*}
    \phi \colon \Delta\cap M \to \R,~~x \mapsto \min\settb{|x-p|^2}{p\in \delta}.
\end{align*}
Since $\phi=0$ on $\delta$ and non-zero everywhere else, we have $\delta\in \PP$. By \cref{smoothDegenThm} we have $\sbb{\Zo(g_\delta)}\neq \sbb{\Zo(g_\alpha)}$ for $\alpha\in \PP\setminus \set{\delta}$. In particular
\begin{equation*}
(-1)^{n+1}\sum_{\substack{\alpha\in\PP \\ \alpha\centernot\subset \partial\Delta}} (-1)^{\dim{\alpha}}\sbb{\Zo(g_\alpha)} \neq \sbb{\Spec{k}}
\end{equation*}
so $\Delta$ is not stably rational.
\end{proof}

\subsection{Monomial Degenerations}

When $\delta$ is not smooth, the degenerations constructed above will typically not be strictly toroidal. However, if we impose some conditions on the torus invariant sections of the ample line bundle, we can obtain the same result for a large class of (possibly singular) polytopes.

\begin{lemma} \label{monomialDegenToroidal}
    Let $X=\P_k(\Delta)$ be a toric variety, $x_I= \prod_{\rho\in I} x_\rho,~I\subset \Sigma^{\Delta}(1)$ be a reduced monomial in the homogeneous coordinate ring of $X$ of degree $\deg(x_I)=[D_\Delta]$ where $\OO_X(D_\Delta)=\Lb(\Delta)$ (cf. \cref{smoothDegeneration}). Similarly, let $g$ be general of degree $[D_\Delta]$. Then
    \begin{equation*}
        \X = Z(gt-x_I)\subset X \times_k R
    \end{equation*}
    is a strictly toroidal $R$-scheme.
\end{lemma}

\begin{proof}
By restricting to a torus invariant open set corresponding to a cone $\sigma$, we may assume that $\X$ is cut out by a regular function $gt-\chi^{m}$ on $U=\Spec R[\sigma^\vee \cap M]$ where $\chi^m$ is a generator for the ideal sheaf of $Z(x_I)\cap U$ (this is possible since $D_\Delta$ is Cartier and $Z(x_I)$ is torus invariant). The scheme $Z(g)$ is a general member of the linear system of $\OO_X(D_\Delta)$, which is base point free. This means that the scheme $\X\cap U$ is cut out by  $\chi^{(0,1,0)}-g$ in the strictly toroidal $R$-scheme 
\begin{equation*}
\Spec{R[N\oplus \N\oplus \N/\sim]/(\chi^{(0,0,1)}-t)}
\end{equation*}
where $(0,1,1)\sim (m,0,0)$.
The claim now follows from \cref{logSubSchTheorem}.
\end{proof}
Let $\Delta$ be a lattice polytope and $\P(\Delta)$ the toric variety. Consider the following assumption on $\Delta$.
\begin{definition} \label{conditionMDef}
Suppose that for every $\rho\in \Sigma^{\Delta}(1)$ there is a monomial $m_\rho$ in the homogeneous coordinate ring of $\P(\Delta)$ such that $\deg(m_\rho)=[D_\Delta]$ and $D_\rho\subset Z(m_\rho)$. Then we say that $\Delta$ (or $\P(\Delta)$) satisfies the condition $(M)$.
\end{definition}
This condition will ensure enough space in the linear systems to construct degenerations that imply strong variation of stable birational types.
\begin{theorem} \label{monomialDegenThm}
    Let $\delta$ be a non-stably rational lattice polytope that satisfies condition $(M)$.
    Then $\delta$ has strong variation of stable birational types.
\end{theorem}

\begin{proof}
    Let $\alpha$ be a lattice polytope such that $\alpha\cap \delta^\circ=\emptyset$ and let $\Delta$ be a lattice polytope containing $\alpha\cup \delta$. Let $g$ be a very general Laurent polynomial with Newton polytope $\Delta$.
    If $\alpha$ is a stably rational polytope, then we are done, so we may assume that $\alpha$ is not stably rational.
    By \cref{fieldReductionProp} it suffices to find a Laurent polynomial $f$ such that $f_\alpha$ and $f_\delta$ are Newton nondegenerate and have Newton polytopes $\alpha$ and $\delta$, respectively, such that 
    \begin{equation*}
    \sbb{\Zo(f_\alpha)}\neq\sbb{\Zo(f_\delta)}.
    \end{equation*}
    Let $F_\eta$ denote a facet of $\delta$ such that $\alpha\cap \delta\subset F_\eta$. Let $x_I=\prod_{\rho\in I} x_\rho$ denote a reduced monomial in the homogeneous coordinate ring of $\P(\delta)$ such that $\eta\in I$. Consider the homogeneous equation (viewed as a Laurent polynomial in the same way as \cref{smoothDegenThm})
    \begin{equation*}
    f = gt - x_I.
    \end{equation*}
    By \cref{monomialLemma} we have $f_\alpha=g_\alpha t$ and $f_\delta=g_\delta t-x_I$ where $g_\delta$ is very general with Newton polytope $\delta$, and $g_\alpha$ very general with Newton polytope $\alpha$. Now, it suffices to show 
    \begin{equation*}
    \Volsb{(Z(f_\alpha))}\neq\Volsb{(Z(f_\delta))}.
    \end{equation*}
    By \cref{monomialDegenToroidal} the closure of $Z(f_\delta)$ in $\P_k(\delta)\times_k R$ is strictly toroidal, so 
    \begin{align*}
    \Volsb(Z(f_\delta))=\sbb{Z(x_I)}=\begin{cases}
        \chi(\delta)\sbb{\Spec{k}} & I=\Sigma^\delta(1) \\
        \sbb{\Spec{k}} & \text{otherwise}
    \end{cases}
    \end{align*}
    which is not equal to $\sbb{\Zo(g_\alpha)}$ since $\alpha$ is not stably rational.
\end{proof}

\begin{theorem} \label{monomialDegenCor}
    Let $\delta\subset \Delta$ be lattice polytopes with $\dim\delta=\dim\Delta$ such that $\delta$ is not stably rational and satisfies the condition $(M)$. Then $\Delta$ is not stably rational.
\end{theorem}
\begin{proof}
Follows using the same argument as \cref{smoothContainment}.
\end{proof}

\begin{example}
    Some examples that satisfy the condition $(M)$ or the second condition of \cref{smoothDegenThm}.
    \begin{itemize}
        \item All dilated standard simplices $d\Delta_n$ and all products
        \begin{equation*}
        d_1\Delta_{n_1}\times \dots\times d_r\Delta_{n_r}.
        \end{equation*}
        Note that $d\Delta_n$ does not satisfy \cref{monomialDegenThm} then $d > n+1$, but they are smooth, and it is easy to show that they satisfy the second condition of \cref{smoothDegenThm}.
        \item The polytopes corresponding to $n$-dimensional weighted projective spaces of the form $\P(d,1,\dots,1)$ for $d\leq n$ with polarization $\OO(d)$.
        \item The polytope $\Delta_{\operatorname{HPT}}$ corresponding to a particular hypersurface in $\P^3\times \P^2$ of bidegree $(2,2)$. See \cref{HPTexample}.
    \end{itemize}
\end{example}

\begin{example}{(Non-simplicial)}
    The condition $(M)$ can hold for non-simplicial toric varieties. The polytope
    \begin{align*}
        \Delta = \Conv\set{0,e_1+e_2,e_1+e_3,e_2+e_3,e_1+e_2+e_3}\subset \R^3
    \end{align*}
    is not simplicial. The homogeneous coordinate ring is generated by $x_0,\dots, x_5$ where
    \begin{align*}
        H^0(\P(\Delta),\Lb(\Delta)) = \langle x_{4}x_{5}^{2},\:x_{2}x_{3}x_{5},\:x_{0}x_{1}x_{4},\:x_{1}x_{3}^{2},\:x_{0}x_{2}^{2} \rangle.
    \end{align*}
    In particular, $\Delta$ satisfies condition (M).
\end{example}

\begin{example}{(Non-Fano)}
    Consider the lattice polytope
    \begin{align*}
        \Delta = \Conv\set{e_1,e_2,2e_1,2e_2,e_1+2e_2,2e_1+e_2,-6e_1-6e_2+2e_3}\subset \R^3.
    \end{align*}
    A \citepalias{M2} computation shows that $\Cl(\P(\Delta))\simeq \Z^4$ and $\Pic(\P(\Delta))\simeq \Z$ with the inclusion $\Z\to \Z^4$ given by $1\mapsto 2e_1$. The anti-canonical divisor $-K_{\P(\Delta)}$ corresponds to $3e_1+e_2-e_4$, in particular it is not $\Q$-Cartier. A simple calculation in \citepalias{M2} shows that $\Delta$ satisfies condition (M).
\end{example}

\begin{example}
    Let $\Delta$ be an empty $n$-dimensional simplex. Then the homogeneous coordinate ring of $\P(\Delta)$ is generated by elements $x_0,\dots,x_n$ such that for each $x_i$ there is some $n_i$ such that $x_i^{n_i}$ corresponds to a global section of $\Lb(\Delta)$. Since $\Delta$ is empty, these generate all global sections of $\Lb(\Delta)$. In particular, $\Delta$ satisfies condition $(M)$ if and only if it is unimodular.
\end{example}

\begin{example} \label{HPTexample}
Consider $\P^3\times \P^2$ with homogeneous coordinates $(U:V:W:T)$ on the first factor and $(x:y:z)$ on the second. Let
\begin{equation*}
    F(x,y,z)=x^2+y^2+z^2-2(xy+xz+zy)
\end{equation*}
and consider the bidegree $(2,2)$ hypersurface defined by
\begin{equation*}
    yzU^2+xzV^2+xyW^2+F(x,y,z)T^2=0.
\end{equation*}
It was shown in \cite{HPT18} that this is not stably rational. The Laurent polynomial obtained by $T=z=1$ has the following Newton polytope.
\begin{align*}
\Delta_{\operatorname{HPT}} = \Conv \set{0, 2e_1, 2e_2, e_2+2e_3, e_1+2e_4, e_1+e_2+2e_5}
\end{align*}
The Laurent polynomial above is not Newton non-degenerate but follows from \cite[Prop. 3.1]{Sch19a} that $\Delta$ is not stably rational.
We have the exact sequence
\begin{equation*}
\begin{tikzcd}
0 \arrow[r] & \Z^5 \arrow[r] & \Z^6 \arrow[r] & \Cl(\P(\Delta)) \arrow[r] & 0
\end{tikzcd}
\end{equation*}
where the first map is given by
\begin{align*}
    e_1&\mapsto 2e_2-2e_5,~  e_2\mapsto 2e_2-2e_4,~ e_3\mapsto e_1-e_2-e_4 \\
    e_4&\mapsto -e_2+e_3-e_5,~ e_5\mapsto -e_4 -e_5+e_6
\end{align*}
There is an isomorphism $\Cl(\P(\Delta))\simeq \Z\times\Z/2\Z\times \Z/2\Z$ such that the Cox ring is $S=k[x_1,\dots,x_6]$ with a $\Z\times\Z/2\Z\times \Z/2\Z$ such that
\begin{align*}
    \deg(x_1)&=(2,1,0),~\deg(x_2)=(1,0,0),~\deg(x_3)=(2,0,1) \\
    \deg(x_4)&=(1,1,0), \deg(x_5)=(1,0,1),~\deg(x_6)=(2,1,1).
\end{align*}
The global sections $H^0(\P(\Delta)), \Lb(\Delta))$ correspond to homogeneous polynomials in $S$ of degree $(4,0,0)$, and a direct calculation yields the following set of generators.
\begin{align*}
\langle x_1^2,x_2^4,x_3^2,x_4^4,x_5^4,x_6^2,x_1x_2x_4,x_4x_5x_6,x_2x_3x_5,x_2^2x_5^2,x_4^2x_5^2,x_4^2x_2^2 \rangle
\end{align*}
This shows that $\Delta_{\operatorname{HPT}}$ satisfies condition (M). 

One can check that any subpolytope of $\Delta$ not properly contained in the boundary is rational, so this is, in a sense, minimal.
\end{example}
In \cref{monomialDegenCor}, we require that the dimensions of the two polytopes be equal. If we start with a polytope $\Delta$ that satisfies condition (M), we cannot immediately use \cref{monomialDegenCor} as a strategy to find non-stably rational polytopes of larger dimensions. However, we can start with $\Delta$ and construct polytopes of higher dimensions that are as ''small'' as possible and contain $\Delta$ (not in the boundary). We give such a construction below and show that these also satisfy the inclusion property of \cref{monomialDegenCor} even if they do not satisfy condition (M) themselves.

\begin{example} \label{exampleHPThigherdim}
Let $\Delta\subset \R^n$ be a lattice polytope. Fix an integer $r\geq 0$ and points $v_k^{+}, v_k^{-}\in \R^n\times \R^r$ for $1\leq k \leq r$ such that $p(v_k^+)=e_k$ and $p(v_k^-)=-e_k$ where $p\colon \R^n\times \R^r\to \R^r$ is the natural projection. We let
\begin{align*}
\Delta(v_1^{+}, v_1^-,\dots, v_r^{+},v_r^-)\subset \R^n\times \R^k
\end{align*}
denote the convex hull of $\Delta$ and the $v_k^{\pm}$. Here, $\Delta\subset \R^n\times \R^r$ is understood via the inclusion $\R^n\to\R^n\times \R^r,~x\mapsto (x,0)$. \cref{fig:subdivision} depicts the polytope $\Delta(e_2,-e_2,e_3,-e_3)\subset \R^3$ where $\Delta=[-1,1]\subset \R$ is the black line.
\end{example}

\begin{figure}
    \centering
    \tdplotsetmaincoords{80}{480}
\begin{tikzpicture}[tdplot_main_coords]	
\fill [green, fill opacity=0.5] (0,0,0) -- (2,0,0) -- (2,2,0) -- cycle;
\fill [green, fill opacity=0.5] (1,1,2) -- (2,0,0) -- (2,2,0) -- cycle;
\fill [green, fill opacity=0.5] (0,0,0) -- (1,1,2) -- (2,0,0) -- cycle;
\fill [blue, fill opacity=0.5] (0,0,0) -- (0,2,0) -- (2,2,0) -- cycle;
\fill [blue, fill opacity=0.5] (1,1,2) -- (0,2,0) -- (2,2,0) -- cycle;
\fill [blue, fill opacity=0.5] (0,0,0) -- (0,2,0) -- (1,1,2) -- cycle;
\fill [red, fill opacity=0.5] (0,0,0) -- (0,2,0) -- (2,2,0) -- cycle;
\fill [red, fill opacity=0.5] (1,1,-2) -- (0,2,0) -- (2,2,0) -- cycle;
\fill [red, fill opacity=0.5] (0,0,0) -- (0,2,0) -- (1,1,-2) -- cycle;
\fill [yellow, fill opacity=0.5] (0,0,0) -- (2,0,0) -- (2,2,0) -- cycle;
\fill [yellow, fill opacity=0.5] (1,1,-2) -- (2,0,0) -- (2,2,0) -- cycle;
\fill [yellow, fill opacity=0.5] (0,0,0) -- (1,1,-2) -- (2,0,0) -- cycle;
\draw (0,0,0) -- (2,0,0) -- (2,2,0) -- (0,2,0) -- cycle;
\draw (1,1,-2) -- (0,0,0);
\draw (1,1,-2) -- (2,0,0);
\draw (1,1,-2) -- (0,2,0);
\draw (1,1,-2) -- (2,2,0);
\draw (1,1,2) -- (0,0,0);
\draw (1,1,2) -- (0,2,0);
\draw (1,1,2) -- (2,2,0);
\draw (1,1,2) -- (2,0,0);
\draw[line width=0.6mm, black] (0,0,0) -- (2,2,0);
\end{tikzpicture}
    \caption{}
    \label{fig:subdivision}
\end{figure}

The polytopes constructed above need not satisfy condition (M) even if $\Delta$ does. However, we will show that they still have the inclusion property as in \cref{monomialDegenCor}.

\begin{proposition} \label{higherDimHPTProp}
Let $\delta\subset \R^n$ be a non-stably rational polytope satisfying condition (M). Let $\delta'=\delta(v_1^+,v_1^-,\dots, v_r^+,v_r^-)\subset\R^{n+r}$. Let $\Delta$ be a polytope of dimension $\dim(\Delta)=\dim(\delta')$ with $\delta'\subset \Delta$. Then $\Delta$ is not stably rational.
\end{proposition}
\begin{proof}
Let $e_1,\dots, e_{n+r}$ denote the standard basis on $\R^n\times \R^r$ with dual basis $e_i^\vee$. Let $H^i_{\leq 0}$ and $H^i_{\geq 0}$ denote the halfspaces defined by $e_{n+i}^\vee\leq 0$ and $e_{n+i}^\vee\geq 0$, respectively. Denote by $H^i_0$ the intersection. Further, let 
\begin{align*}
\alpha^i_{\leq 0} &= \delta\cap H^1_{\leq 0}\cap \dots \cap H^i_{\leq 0} \\
\alpha^i_{\geq 0} &= \delta\cap H^1_{\geq 0}\cap \dots \cap H^i_{\geq 0} \\ 
\alpha^i_{0} &= \alpha^i_{\leq 0}\cap \alpha^i_{\geq 0}
\end{align*}
For convenience we let $\alpha^0_0$ denote $\alpha^1_{\leq 0}\cup \alpha^1_{\geq 0}$. Note that $\alpha^r_0=\delta$. Now consider the regular subdivision $\PP$ on $\Delta$ defined by
 \begin{align*}
\phi\colon \Delta \to \R,~~x\mapsto \sum_{i=0}^r \min_{p\in \alpha^i_0} |x-p|^2.
\end{align*}
\cref{fig:subdivision} shows this subdivision in the case where $\Delta(e_2,-e_2,e_3,-e_3)$ and $\Delta=[-1,1]\subset \R$.

The function $\phi$ is identically zero on (and only on)  $\alpha^r_0$, so $\delta\in\PP$. It is not differentiable along the boundary of $\delta'$. Hence any cell containing $\delta$ must be properly contained in $\delta'$. These are of width 1, hence rational. Since $\delta$ has strong variation of stable birational type, it does not cancel with any cell in $\PP$ that does not properly contain it by \cref{monomialDegenThm}. \cref{tropicalVolume} now shows that $\Delta$ is not stably rational.
\end{proof}

\subsection{Polytopes with positive Kodaira dimension has strong variation}
A consequence of \cref{monomialDegenThm} is that polytopes $\Delta$ with $\kappa(\Delta)\geq 0$ have strong variation. For example, if $F$ is a general homogeneous polynomial of degree $n+1$ in $\P^n$, then the stable birational type of $F+tx_0\cdots x_n$ varies with $t$.

\begin{corollary} \label{CYvariation}
    Let $\delta$ be a lattice polytope of dimension at least two such that
    \begin{align*}
        \delta^\circ\cap M = \set{0}.
    \end{align*}
   Then $\delta$ has strong variation of stable birational type.
\end{corollary}
\begin{proof}
The unique lattice point in the interior of $\delta$ induces a monomial with a reduced zero set containing the whole toric boundary, so $\delta$ satisfying condition $(M)$. The result now follows from \cref{monomialDegenThm}.
\end{proof}

\begin{corollary}
    Let $\delta$ be a polytope of dimension at least 2 with $|\delta^\circ\cap M|\neq\emptyset$. Then $\delta$ has strong variation of stable birational type.
\end{corollary}
\begin{proof}
    We proceed by induction on $|\delta^\circ\cap M|=k$. The case $k=1$ follows from \cref{CYvariation}. Suppose the claim holds for $|\delta^\circ\cap M|<k$.

    Let $\alpha$ be any lattice polytope such that $\alpha\cap \delta^\circ=\emptyset$. Let $\Delta$ be any lattice polytope containing $\alpha\cup\delta$ and $g$ a very general Laurent polynomial with Newton polytope $\Delta$.
    Let $0\in\delta^\circ$ be a lattice point, and $\PP$ is the subdivision of $\delta$ obtained by pulling $0$ in the trivial subdivision. This consists of the convex hull of $0$ with the faces of $\delta$ and is regular (see \cref{pullingSubDivEx}). The restriction of $\PP$ to $\alpha\cap\delta$ is constant, so if for a very general Laurent polynomial $g$ with Newton polytope $\Delta$, we have $\sbb{\Zo(g_\alpha)} = \sbb{\Zo(g_\delta)}$. Then we also have 
    \begin{equation} \label{contradictionEq}
        \sbb{\Zo(g_\alpha)} = (-1)^{n+1}\sum_{\substack{\beta\in\PP \\ \beta\centernot\subset \partial\delta}} (-1)^{\dim{\beta}}\sbb{\Zo(g_\beta)}.
    \end{equation}
    The right-hand side of this equation contains a polytope $\beta_0$ with at least one and at most $k-1$ lattice points in its interior. By induction, it is not stably birational to any other lattice polytopes not properly containing it, assuming $\dim \beta_0\geq 2$. If $\beta_0$ has dimension 1 and contains $r-1$ interior lattice points, then $\sbb{\Zo(g_\beta)}=r\sbb{\Spec{k}}$. The faces of $\PP$ of dimension at least one not contained in $\partial \delta$ bijectively correspond to the faces of $\delta$. If all interior lattice points lie on edges, then the signed number of terms on the right-hand side of \eqref{contradictionEq} is
    \begin{align*}
    (-1)^{n}(-k+1) + (-1)^{n+1}\sum_{\substack{\beta\in\PP\\ \beta\centernot\subset \partial\delta}} (-1)^{\dim{\beta}} &= \chi(\delta)+ (-1)^{n+1}( k-1 ) 
    \end{align*}
    This is $1-k$ if $n$ is odd and $1+k$ if $n$ is even.
    Since $k\geq 2$, this is never equal to 1. \eqref{contradictionEq} now implies that $\Zo(g_\alpha)$ is not irreducible, which is a contradiction.
    
    If $\beta_0$ is contained in a polytope without interior lattice points, these are not stably birational as they have different Kodaira dimensions. If it is contained in a lattice polytope with interior lattice points, then this is a lattice polytope with at least 1 and at most $k-1$ interior lattice points, and hence it is not stably birational to $\beta_0$ by induction. This shows that \eqref{contradictionEq} cannot hold a contradiction. 
\end{proof}

\section{Four dimensional polytopes} \label{section4}
When the polytopes have dimension $\leq 4$, we prove that the property of not being stably rational is inclusion preserving without any assumptions on the polytopes. Throughout this section, we fix $k=\C$.
\begin{theorem} \label{mainThmDim4}
    Let $\delta\subset\Delta\subset M_\R$ be lattice polytopes with $\dim\Delta \leq 4$, $\delta$ not stably rational, $\delta\not\subset \partial\Delta$. Then $\Delta$ is also not stably rational.
\end{theorem}
We show this by showing the following stronger result.
\begin{proposition} \label{ratTheorem}
        Let $\Delta\subset M_\R$ be a lattice polytope of dimension $\leq 3$. If $\kappa(\Delta)=-\infty$, then $\Delta$ is rational.
\end{proposition}

Let $X$ be an irreducible variety over $\C$. Recall that $X$ has \emph{$\LL$-rational} singularities if there exists a resolution $\pi\colon \widetilde{X}\to X$ such that $[\pi^{-1}(x)]=1 \mod{\LL}$ in $\Kvar{k(x)}$ for all $x\in X$.
Strictly toroidal singularities are $\LL$-rational \cite[4.2.6]{NS19}, in particular, the hypersurfaces $Z(f)$ have $\LL$-rational singularities when $f$ is general with respect to its Newton polytope.

For a quasi-projective variety $Z$, let $H_c^{k}(Z)$ denote cohomology with compact support. These cohomology groups carry mixed Hodge structures \cite{DK87}, so one can define Hodge numbers $h^{p,q}(H^k_c(Z))$ and consequently the $(p,q)$-Euler characteristic
\begin{equation*}
    e^{p,q}(Z)=\sum_{k} (-1)^k h^{p,q}(H^k_c(Z)).
\end{equation*} 
For any decomposition of $Z$ into a disjoint union $\cup_i Z_i$ of locally closed sets $Z_i$, we have $e^{p,q}(Z)=\sum_{i} e^{p,q}(Z_i)$
By \cite[]{DK87}. Moreover, when $Z$ is smooth and proper, one has $e^{p,q}(Z)=(-1)^{p+q}h^{p,q}(Z)$. Define the $E$-polynomial
\begin{align*}
    E\colon \Kvar{\C} &\to \Z[u,v]
\end{align*}
by $X\mapsto E(X)=\sum e^{p,q} (X)u^pv^q$ for smooth proper varieties $X$. Since $\Kvar{\C}$ is generated by proper smooth varieties, $E$ extends uniquely to all $\Kvar{\C}$.
\begin{lemma}\label{hodge_number_lemma}
    Let $X$ be an irreducible projective variety over $\C$ with $\LL$-rational singularities. Then for any smooth birational model $\widetilde{X}$ of $X$ we have $e^{p,0}(X)=e^{p,0}(\widetilde{X})$.
\end{lemma}
\begin{proof}
First, note that $E(\LL)=uv$ where $\LL$ denotes the class of the affine line. Since $X$ has $\LL$ rational singularities we have $[X]=[\widetilde{X}] \mod \LL$ and since $E$ is additive we see that the difference between $E(X)$ and $E(\widetilde{X})$ only occurs in the coefficients of $u^pv^q$ for $p,q>0$. In particular, we have $e^{p,0}(X)=e^{p,0}(\widetilde{X})$ for all $p$. 
\end{proof}
Let $\Delta\subset M_\R$ be a lattice polytope and $f$ a general Laurent polynomial with Newton polytope $\Delta$. The numbers $e^{p,0}(\Zo(f))$ can be computed from the polytope as follows.
\begin{proposition}{{\cite[Prop. 5.8]{DK87}}} \label{hodge_comp}
    Let $\Delta\subset M_\R$ be a lattice polytope and $f$ general with Newton polytope $\Delta$. Then
    \begin{equation} \label{compute_plurigen}
        e^{p,0}(\Zo(f))= (-1)^{n-1} \sum_{\dim{\alpha} = p+1} |\alpha^\circ\cap M|
    \end{equation}
    where the sum is taken over all faces $\alpha\subset \Delta$.
\end{proposition}
By additivity, this gives a formula for the numbers $e^{p,0}(Z(f))$.
\begin{corollary} \label{hdgNumberCor}
    Let $f$ be a general Laurent polynomial with Newton polytope $\Delta$ of dimension $n+1$. Then $h^{p,0}(Z(f))=0$ for $0<p<n$ and $h^{n,0}(Z(f))=|\Delta^\circ\cap M|$.
    \end{corollary}
    \begin{proof}
    Using additivity, we arrive at the following formula
    \begin{align*}
        e^{p,0}(Z(f)) &= \sum_{i=0}^n\sum_{\substack{\delta\subset \Delta \\ \dim{\delta}=i}} \sum_{\substack{\alpha\subset \delta \\ \dim{\alpha}=p+1}} (-1)^{\dim{\delta-1}}|\alpha^\circ \cap M|
    \end{align*}
    Consider a face $\alpha\subset\Delta$ of dimension $p+1<n$. For every proper face $\delta\subset \Delta$ containing $\alpha$, we get a signed contribution of $|\alpha^\circ \cap M|$ to the sum above. This means that up to a sign, the contribution of $|\alpha^\circ\cap M|$ amounts to
    \begin{equation*}
        \sum_{\substack{\alpha\subset\delta}} (-1)^{\dim{\delta}+1} |\alpha^\circ\cap M|.
    \end{equation*}
    Up to a sign, this equals $(\chi(\operatorname{Star}(\alpha))-1)|\alpha^\circ \cap M|$ where $\operatorname{Star}(\alpha)$ is the star of $\alpha$ as a CW complex, consisting of cells containing $\alpha$. Since this is contractible, the sum above is equal to 0.
    \end{proof}

    We are now ready to prove the following theorem.

    \begin{proposition} \label{dim3_rational}
        Let $\Delta\subset M_\R$ be a lattice polytope of dimension $\leq 3$ and let $f$ be a general Laurent polynomial with Newton polytope $\Delta$. Assume further that $\Delta^{\FI}=\emptyset$. Then $\Zo(f)$ is rational.
    \end{proposition}
    \begin{proof}
        The $1$ dimensional case is obvious. If $\dim \Delta = 2$ and $\Delta$ have no interior lattice points, then it is equivalent to a polytope contained in $2\Delta_2$ and hence rational. If $\dim\Delta = 3$, $Z(f)$ is (after a resolution of singularities) a ruled surface. If $Z(f)$ is birational to a surface $C\times \P^1$ then $h^{1,0}(Z(f))=g$ for $g$ the genus of $C$. By \cref{hdgNumberCor}, we have $g=0$, so $Z(f)$ is a rational surface.
    \end{proof}

    \begin{theorem} \label{dim4PolytopeTheorem}
    Let $\delta\subset\Delta\subset M_\R$ be a lattice polytopes with $\dim\Delta \leq 4$ and $\delta$ non-stably rational, $\delta\not\subset \partial\Delta$. Then $\Delta$ is not stably rational.
    \end{theorem}
    \begin{proof}
        If $\delta^{\FI}\neq \emptyset$ then $\Delta^{\FI}\neq \emptyset$ so we may assume that $\delta^{\FI}$ is empty. Consider the regular subdivision $\PP$ given by the lower convex envelope of
        \begin{align*}
            \phi \colon \Delta\cap M \to \R,~~ x \mapsto \min\sett{|x-p|^2}{p\in \delta}.
        \end{align*}

        Since $\phi=0$ on $\delta$ and non-zero everywhere else, we have $\delta\in \PP$. Furthermore, by \cref{dim3_rational}, the only non-stably rational polytopes appearing in $\PP$ are of dimension 4. This means that

        \begin{equation*}
            (-1)^{\dim{\Delta}}\sum_{\substack{\alpha\in\PP \\ \alpha\centernot\subset \partial\Delta}} (-1)^{\dim{\alpha}}\sbb{\Zo(g_\alpha)} \neq \sbb{\Spec{k}}
        \end{equation*}
        so $\Delta$ is not stably rational by \cref{tropicalVolume}.
    \end{proof}

\section{Applications} \label{section5}
\subsection{Products of projective spaces}
We start by showing how to use $\Delta_{\operatorname{HPT}}$ from \cref{HPTexample} to prove that a very general $(2,3)$ divisor in $\P^3\times \P^4$ is not stably rational. A priori, the fact that a very general $(2,2)$ divisor in $\P^3\times \P^2$ is not stably rational implies that a very general $(2,4)$ in $\P^3\times \P^4$ is not stably rational \cite[Theorem 4.5]{NO20a}. The fact that $\Delta_{\operatorname{HPT}}$ is special allows us to increase dimension without increasing the degree. The idea is to embed $\Delta_{\operatorname{HPT}}$ into $\R^7$ and construct a ''double cone '' over it (cf. \cref{fig:subdivision}). We obtain the following result.
\begin{theorem} \label{23Divisor}
A very general hypersurface of bidegree $(2,3)$ in $\P^3\times \P^4$ or $(3,3)$ in $\P^4\times \P^4$ is not stably rational.
\end{theorem}
\begin{proof}
Consider the polytope
\begin{align*}
\Delta=\Conv \{0, 2e_1,2e_2,e_2+2e_3,e_1+2e_4,e_1+e_2+2e_5, \\
            e_1+e_4+e_6,e_1-e_6,e_1+e_5+e_7,e_1-e_7\}.
\end{align*}
Note that the first 6 points define $\Delta_{\operatorname{HPT}}$. It is properly contained in $\Delta$, explicitly realized as the intersection of $\Delta$ with the hyperplanes $e_6^\vee=e_7^\vee=0$ (standard dual basis). The polytope $\Delta$ is not stably rational, and any polytope containing it is also not stably rational by \cref{higherDimHPTProp}. It suffices to show that this is unimodular equivalent to a polytope contained in $2\Delta_3\times 3\Delta_4$. After the unimodular affine transformation given by
\begin{align*}
e_1\mapsto e_5,~e_2\mapsto e_4,~e_3\mapsto e_3,~e_4\mapsto e_2-e_5+e_6,\\
e_5\mapsto e_1-e_5+e_7,~e_6\mapsto e_5-e_6,~e_7\mapsto e_5-e_7
\end{align*}
followed by a translation by $e_1$ we find the above to be equivalent to
\begin{align*}
P=\Conv \{&e_5,3e_5,e_2+2e_5,e_5+e_6,e_1+2e_5,e_5+e_7,\\
           &2e_4+e_5,2e_2+2e_6,2e_1+e_4+2e_7,2e_3+e_4+e_5 \}
\end{align*}
The polytope $P$ is contained in $2\Delta_3\times 3\Delta_4$, so the latter is not stably rational by \cref{higherDimHPTProp}. It follows from \cite[Theorem 4.5]{NO20a} that $3\Delta_4\times 3\Delta_4$ is also not stably rational.
\end{proof}

\subsection{Hypersurfaces in projective space}
One expects that a very general hypersurface of degree $d\geq 3$ in $\P^{N+1}$ is not stably rational when $N\geq 3$. The best bound where this is known is due to Schreieder. He proved that a very general hypersurface of degree $d\geq \log_2(N)+2$ in $\P^{N+1}$ is not stably rational \cite{Sch19a}. Except for the quartic fivefold, this covers all known cases of non-stably rational hypersurfaces in projective space. The proof is based on an explicit construction of a special hypersurface that does not admit a decomposition of the diagonal. These hypersurfaces are special in the sense that they have small Newton polytopes compared to a general hypersurface with the same degree and dimension. We use the special nature of these to improve the bound given in \cite{Sch19a} when the ground field is $\C$.

\begin{theorem} \label{hypersurfacesTheorem}
Let $N\geq 3$ and $X\subset \P^{N+1}$ be a very general hypersurface of degree $d\geq n+2$ where $N=n+r$ and $2^{n-1}-2 \leq r \leq 2^n-2+2^{n-2}(n-1)$. Then $X$ is not stably rational.
\end{theorem}

The first new case we obtain is the non-stable rationality of a quintic hypersurface in $\P^{N+1}$ for $N=10,11,12,13$. Furthermore, very general sextic hypersurfaces in $\P^{N+1}$ are not stably rational for $19\leq N \leq 30$.

\begin{figure}[!ht]
    \centering
    \begin{tikzpicture}[scale=0.58, transform shape]
    \tikzmath{
    \xmax = 19;
    \ymax = 5;
    \xmaxp = \xmax+0.5;
    \ymaxp = \ymax+0.5;
    \xmaxpp = \xmax+1;
    \ymaxpp = \ymax+1;
    \xmaxtwo = 2*\xmax - 4;
    \ymaxtwo = 2*\ymax;
    \xmaxtwoplusone = (\xmaxtwo + 2)/2 - 1/3;
    \xmaxtwoplustwo = (\xmaxtwo + 3)/2 - 0.25;
    \xmaxtwoplusthree = (\xmaxtwo + 4)/2 - 0.25;
}
\draw[->] (0,0) -- (\xmaxp,0);
\draw[->] (0,0) -- (0,\ymaxp);
\node at (\xmaxpp,0) {$N$};
\node at (0,\ymaxpp) {$d$};
\foreach \x in {1,...,\xmaxtwo}{
    \tikzmath{\xx = \x/2;}
    \draw[fill=black] (\xx,0) circle (1.5pt);
}
\foreach \x in {1,...,\xmaxtwo}{
    \tikzmath{
        int \xx;
        \xx=\x + 2;
        \xxx = \x/2;
        }
    \draw (\xxx,-0.5) node {\xx};
}
\foreach \y in {1,...,\ymax}{
    \draw[fill=black] (0,\y) circle (1.5pt);
}
\foreach \y in {1,...,\ymax}{
    \tikzmath{
        int \yy;
        \yy=\y + 2;
        }
    \draw (-0.5,\y) node {\yy};
}
%Cubics
\node[diamond, mark size=2pt,draw=black, fill=yellow, line width=1pt] at (1/2,1) {};
\foreach \x in {2,...,\xmaxtwo}{
    \tikzmath{
        int \xx;
        \xx=\x + 2;
        \xxx = \x/2;
        }
    \node[circle, mark size=2.5pt,draw=black, line width=1pt] at (\xxx,1) {};
}

%Quartics
\draw[fill=red] (1/2,2) circle (3pt);
\draw[fill=red] (1,2) circle (3pt);
\node[circle, mark size=2.5pt,draw=black, fill=red, line width=1pt] at (1/2,2) {};
\node[circle, mark size=2.5pt,draw=black, fill=red, line width=1pt] at (1,2) {};
\node[star, mark size=2.5pt,draw=black, fill=green, line width=1pt] at (3/2,2) {};
\foreach \x in {4,...,\xmaxtwo}{
    \tikzmath{
        int \xx;
        \xx=\x + 2;
        \xxx = \x/2;
        }
    \node[circle, mark size=2.5pt,draw=black, line width=1pt] at (\xxx,2) {};
}

%Quintics
\foreach \x in {1,...,7}{
    \tikzmath{
        int \xx;
        \xx=\x + 2;
        \xxx = \x/2;
        }
    \node[circle, mark size=2.5pt,draw=black, fill=red, line width=1pt] at (\xxx,3) {};
}
\foreach \x in {8,...,11}{
    \tikzmath{
        int \xx;
        \xx=\x + 2;
        \xxx = \x/2;
        }
    \node[rectangle, scale=1.3,draw=black, fill=blue, line width=1pt] at (\xxx,3) {};
}
\foreach \x in {12,...,\xmaxtwo}{
    \tikzmath{
        int \xx;
        \xx=\x + 2;
        \xxx = \x/2;
        }
    \node[circle, mark size=2.5pt,draw=black, line width=1pt] at (\xxx,3) {};
}

%Sextics
\foreach \x in {1,...,16}{
    \tikzmath{
        int \xx;
        \xx=\x + 2;
        \xxx = \x/2;
        }
    \node[circle, mark size=2.5pt,draw=black, fill=red, line width=1pt] at (\xxx,4) {};
}
\foreach \x in {17,...,28}{
    \tikzmath{
        int \xx;
        \xx=\x + 2;
        \xxx = \x/2;
        }
    \node[rectangle, scale=1.3,draw=black, fill=blue, line width=1pt] at (\xxx,4) {};
}
\foreach \x in {29,...,\xmaxtwo}{
    \tikzmath{
        int \xx;
        \xx=\x + 2;
        \xxx = \x/2;
        }
    \node[circle, mark size=2.5pt,draw=black, line width=1pt] at (\xxx,4) {};
}

%Deg 7
\foreach \x in {1,...,33}{
    \tikzmath{
        int \xx;
        \xx=\x + 2;
        \xxx = \x/2;
        }
    \node[circle, mark size=2.5pt,draw=black, fill=red, line width=1pt] at (\xxx,5) {};
}
\foreach \x in {34,...,\xmaxtwo}{
    \tikzmath{
        int \xx;
        \xx=\x + 2;
        \xxx = \x/2;
        }
    \node[rectangle, scale=1.3,draw=black, fill=blue, line width=1pt] at (\xxx,5) {};
}

\node[scale=1.5] at (\xmaxtwoplusone,5) {$\cdots$}; 
\node[scale=1.5] at (\xmaxtwoplusone,4) {$\cdots$}; 
\node[scale=1.5] at (\xmaxtwoplusone,3) {$\cdots$}; 
\node[scale=1.5] at (\xmaxtwoplusone,2) {$\cdots$}; 
\node[scale=1.5] at (\xmaxtwoplusone,1) {$\cdots$}; 
\node[scale=1.5] at (\xmaxtwoplusone,-1/2) {$\cdots$}; 

\node[rectangle, scale=1.3,draw=black, fill=blue, line width=1pt] at (\xmaxtwoplustwo,5) {};

\node[circle, mark size=2.5pt,draw=black, line width=1pt] at (\xmaxtwoplustwo,4) {};
\node[circle, mark size=2.5pt,draw=black, line width=1pt] at (\xmaxtwoplustwo,3) {};
\node[circle, mark size=2.5pt,draw=black, line width=1pt] at (\xmaxtwoplustwo,2) {};
\node[circle, mark size=2.5pt,draw=black, line width=1pt] at (\xmaxtwoplustwo,1) {};
\node at (\xmaxtwoplustwo,-1/2) {67};

\node[circle, mark size=2.5pt,draw=black, line width=1pt] at (\xmaxtwoplusthree,5) {};
\node[circle, mark size=2.5pt,draw=black, line width=1pt] at (\xmaxtwoplusthree,4) {};
\node[circle, mark size=2.5pt,draw=black, line width=1pt] at (\xmaxtwoplusthree,3) {};
\node[circle, mark size=2.5pt,draw=black, line width=1pt] at (\xmaxtwoplusthree,2) {};
\node[circle, mark size=2.5pt,draw=black, line width=1pt] at (\xmaxtwoplusthree,1) {};
\node at (\xmaxtwoplusthree,-1/2) {68};
\end{tikzpicture} \caption{\protect\tikz\protect\node[circle,scale=0.6,draw=black, fill=red, line width=1pt] {}; not stably rational by \cite{Sch19a}. \protect\tikz\protect\node[rectangle,scale=0.6,draw=black, fill=blue, line width=1pt] {}; not stably rational by \cref{hypersurfacesTheorem}. \protect\tikz\protect\node[diamond,scale=0.5,draw=black, fill=yellow, line width=1pt] {}; is irrational by \cite{CG72}. \protect\tikz\protect\node[star,scale=0.5,draw=black, fill=green, line width=1pt] {}; is not stably rational by \cite{NO20a}.}
    \label{fig:irrationalHypersurfaces}
\end{figure}

\begin{remark} \label{degreeRemark}
Let $r=2^n-2+2^{n-2}(n-1)$ with $n\geq 3$ and write $N=n'+r'$ with $2^{n'-1}-2\leq r' \leq 2^{n'}-2$. 
We claim that
\begin{equation*}
\log_2(n+3) -1 > n'-n \geq \log_2(n+3)-2.
\end{equation*}
By the property of $r'$, we have
\begin{equation*}
2^{n'-1}+n' \leq 2^{n-2}(n+3)+n \leq 2^{n'}+n'.
\end{equation*}
The left inequality gives
\begin{align*}
n+3 \geq 2^{n'-n+1} + 2^{2-n}(n'-n) > 2^{n'-n+1}
\end{align*}
so $\log_2(n+3)-1>n'-n$. The other inequality implies
\begin{align*}
n+3 \leq 2^{n'-n+2} + 2^{2-n}(n'-n).
\end{align*}
Since $n'-n < \log_2(n+3)-1$ we have $2^{2-n}(n'-n)<1$ so $n+3\geq 2^{n'-n+2}$, hence $n'-n\geq \log_2(n+3)-2$. This shows that the difference between our bound and the one in \cite{Sch19a} grows logarithmically in $n$.
\end{remark}

Let $N\geq 3$ and let $r\geq 1, n\geq 2$ be integers such that $2^{n-1}-2 \geq r \geq 2^n-2$ and $N=n+r$. Let $d\geq n+2$. For $\epsilon\in \set{0,1}^n$ let $|\epsilon |=\sum_{i=1}^n \epsilon_i$ and consider a bijection
\begin{align*}
\rho \colon \set{0,1}^n &\to \set{1,\dots, 2^n}
\end{align*}
such that $\rho^{-1}(1)=(0,\dots, 0)$ and $\rho^{-1}(2)=(1,\dots, 1)$. Let $S=\rho^{-1}(\set{1,\dots, r+1})$. Define polynomials
\begin{align*}
R_{\operatorname{ev}} &= -x_0^{d-n}x_1\cdots x_n+ \sum_{\epsilon\in S} x_0^{d-2-|\epsilon|}x_1^{\epsilon_1}\cdots x_n^{\epsilon_n} x_{n+\rho(\epsilon)}^2 \\
R_{\operatorname{odd}} &=-x_0^{d-n+1}x_2\cdots x_n + \sum_{\epsilon\in S} x_0^{d-2-|\epsilon| + (-1)^{1-\epsilon_1}}x_1^{1-\epsilon_1}x_2^{\epsilon_2}\cdots x_n^{\epsilon_n} x_{n+\rho(\epsilon)}^2.
\end{align*}
The polynomials $F_d$ defined as
\begin{align*}
F_d = \begin{cases}
t^2\left ( \sum_{i=0}^n x_i^{d/2} \right )^2 + R_{\operatorname{ev}} & d ~\text{even} \\
t^2\left ( x_1\sum_{i=0}^n x_i^{(d-1)/2} \right )^2 + R_{\operatorname{odd}}& d ~\text{odd}.
\end{cases}
\end{align*}
define hypersurfaces $X_d$ that are not stably rational for very general $t\in \C$ by \cite{Sch19a}. By \cite[Prop. 3.1]{Sch19a} any hypersurface in the given projective space degenerating to $X$ is also not stably rational. This means that any polytope contained in $d\Delta_{N+1}$ and contains the Newton polytope of $F_d$ is not stably rational. When $d$ is even, the Newton polytope of $F_d$ is given by the convex hull of the following vectors (here we are considering the affine open $x_1\neq 0$)
\begin{align*}
0, de_1, \dots, de_n,~~(d-n)e_1+\sum_{i=2}^n\epsilon_ie_i,~~(d-2-|\epsilon|)e_1 + \sum_{i=2}^n\epsilon_ie_i+ 2e_{n+\rho(\epsilon)}, \epsilon\in S
\end{align*}
and when $d$ is odd, it is given by
\begin{align*}
&0, de_1, \dots, de_n,~~(d-n+1)e_1+\sum_{i=2}^n\epsilon_ie_i \\
&(d-2-|\epsilon|+(-1)^{1-\epsilon_1})e_1 +\sum_{i=2}^n\epsilon_ie_i + 2e_{n+\rho(\epsilon)}, \epsilon\in S.
\end{align*}
We will modify these slightly. First, we restrict to the case where $r=2^n-2$. It is in this case that we get the strongest bounds. We choose our bijection such that $\rho^{-1}(2^n)=(0,1,\dots,1)$ so $S=\rho^{-1}(1,\dots, 2^n-1)$ consists of all $n$-tuples except $(0,1,\dots,1)$. We consider polynomials of the form $F_d+G$ where $G$ is general in the variables $x_0,\dots, x_n$ and $x_{n+\rho(\epsilon)}$ for $\epsilon$ satisfying $d-2-|\epsilon|\leq 2$ in the even case and $d-2-|\epsilon|+(-1)^{1-\epsilon_1}\leq 2$ in the odd case. The Newton polytope (again in the open affine $x_1\neq 0$) contains the Newton polytope of $F_d$, so by \cite[Prop. 3.1]{Sch19a} these are not stably rational. After possibly changing the order of rows and columns and performing a change of coordinates in the torus given by
\begin{equation*}
x_{n+\rho(\epsilon)} \mapsto \begin{cases}
x_0^{-\floor{\frac{d-2-|\epsilon|}{2}}}x_{n+\rho(\epsilon)} &d~\text{odd} \\
 x_0^{-\floor{\frac{d-2-|\epsilon|+(-1)^{1-\epsilon_1}}{2}}}x_{n+\rho(\epsilon)} & d~\text{even}
\end{cases}
\end{equation*}
we obtain non-stably rational polytopes $\Delta(n)\subset \R^{n+2^n-1}=\R^{N+1}$ given by the convex hull of the following vectors.
\begin{align*}
&0, de_1,\dots, de_{2n},\epsilon_1e_1+\dots \epsilon_n e_n + 2e_{2n+\rho'(\epsilon)},~\epsilon\in S'
\end{align*}
where $|\epsilon|\leq n-2$ and $\rho'\colon \set{1,\dots, 2^n-n-1}\to S'$ is a bijection to the set $S'$ of such $\epsilon$.

\begin{proposition} \label{conditionMprop}
    The polytopes $\Delta(n)$ satisfies condition $(M)$.
\end{proposition}
\begin{proof}
The proof is based on an explicit computation of the class group of $\P(\Delta(n))$ similar to \cref{HPTexample}.
Note that $\Delta(n)$ is a simplex, so it has exactly $2^n+n$ supporting halfspaces. A direct computation shows that the normal directions to these are given by
\begin{align*}
&2e_i^\vee -\sum_{\epsilon\in S'} \epsilon_i e^\vee_{2n+\rho'(\epsilon)}\geq 0, \quad i=1,\dots,n \\
&e^\vee_i\geq 0, \quad i=n+1,\dots, 2^n+n-1  \\
&-\sum_{i=1}^{2n} 2e^\vee_i - \sum_{\epsilon\in S'} (d-|\epsilon|) e^\vee_{2n+\rho'(\epsilon)}\leq 2d. 
\end{align*}
Here, $e_1,\dots, e_{2^n+n-1}$ is the standard basis on $\Z^{2^n+n-1}$ and $e_i^\vee$ denotes the dual basis. Let $f_1,\dots, f_{2^n+n}$ be the standard basis on $\Z^{2^n+n}$. We have an exact sequence
\begin{equation*}
\begin{tikzcd}
0 \arrow[r] & \Z^{2^n+n-1} \arrow[r] & \Z^{2^n+n} \arrow[r] & \Cl\P(\Delta(n)) \arrow[r] & 0
\end{tikzcd}
\end{equation*}
The first map is given as follows.
\begin{align*}
e_i&\mapsto 2f_i - 2f_{2^n+n}, \quad i=1,\dots, n \\
e_i&\mapsto f_i - 2f_{2^n+n}, \quad i=n+1,\dots, 2n \\
e_{2n+\rho'(\epsilon)} &\mapsto f_{2n+\rho'(\epsilon)} - \sum_{i=1}^n \epsilon_if_i - (d-|\epsilon|)f_{2^n+n}, \quad \epsilon\in S'
\end{align*}
We may identify the class group with $\left (\Z/2\Z\right )^n\times \Z$ where $(0,\dots, 0,2d)$ is the image of the ample divisor class induced by $\Delta(n)$. It is useful to have some notation for the map $\Z^{2^n+n}\to \left (\Z/2\Z\right )^n\times \Z$ that induces an isomorphism with the class group. Let $a_1,\dots, a_n,b$ be a basis for $\left (\Z/2\Z\right )^n\times \Z$. The map is given by
\begin{align*}
\Z^{2^n+n}&\to \left (\Z/2\Z\right )^n\times \Z \\
f_i&\mapsto a_i + b,\quad i=1,\dots, n \\
f_{j}&\mapsto 2b,\quad j=n+1,\dots, 2n \\
f_{2n+\rho'(\epsilon)} &\mapsto \sum_{i=1}^n \epsilon_i a_i + d b ,\quad \epsilon\in S' \\
f_{2^n+n}&\mapsto b.
\end{align*}
It suffices now to see that each $f_i$ occurs in a reduced sum of elements that map to $2db$ via this map. For $\epsilon\in S'$ let $I_\epsilon$ denote the indices $i$ such that $\epsilon_i=1$. Then
\begin{equation*}
f_{2n+\rho'(\epsilon)} + \sum_{i\in I_\epsilon} f_i
\end{equation*}
maps to $(d+|\epsilon|)b$ so contains no torsion part. By adding a suitable subset of $f_{n+1},\dots, f_{2n}$ (and a copy of $f_{2^n+n}$ if $d+|\epsilon|$ is odd), we obtain a reduced sum of elements that correspond to $2db$. All $f_i$ occur in one of these, so $\Delta(n)$ satisfies condition (M).
\end{proof}

\begin{remark} \label{lowerRvalueRemark}
By \cite{Sch19a} the polytopes in $\R^{n+r+1}$ given by removing all but $r-2$ of the vertices
\begin{align} \label{epsilonElements}
&d e_{n+2},\dots, d e_{2n} \\
&\epsilon_1e_1+\dots \epsilon_n e_n + 2e_{2n+\rho'(\epsilon)},~\epsilon\in S'\setminus \set{(0,\dots,0)}
\end{align}
are non-stably rational when $2^{n-1}-1\leq r \leq 2^{n}-2$. These polytopes need not satisfy condition (M), but they will if we choose $\rho$ carefully. The computation of the class group in \cref{conditionMprop} is valid for any $r$ and without conditions on the $r-2$ vertices removed, but for any vertex removed from $\Delta(n)$ the corresponding $f_i$ is removed in the presentation of the class group. Now the conclusion of \cref{conditionMprop} holds if we remove the vertices that do not correspond to the following elements.
\begin{align*}
&f_i,~i=1,\dots, n, \\
&f_j,~j=n+1,\dots, 2n, \\
&f_{2n+\rho'(\epsilon)},~ \epsilon=(0,\dots,0,1,0,\dots, 0) \\
&f_{2^n+n}
\end{align*}
By choosing the bijection carefully, we can achieve this for any $2^{n-1}-2\leq r \leq 2^n-2$ when $n\geq 6$. 
\end{remark}

Using the technique introduced in \cref{exampleHPThigherdim}, we can construct higher dimensional polytopes containing $\Delta(n)$ such that any polytope containing it is not stably rational. This construction can be used to construct polytopes $\Delta(n,r)$ in $\R^{2^n+n-1}\times \R^r$ for $r\leq 2^{n-2}(n-1)$ without increasing the total degree, meaning that $\Delta(n,r)\subset (n+2)\Delta_{2^n+n-1+r}$.

\begin{lemma} \label{sumLemma}
Let $S'$ be the set of $n$-tuples $\epsilon$ with $|\epsilon|\leq n-2$. Then
\begin{equation*}
\sum_{\epsilon\in S'} \floor{\frac{n-|\epsilon|}{2}}=2^{n-2}(n-1).
\end{equation*}
\end{lemma}
\begin{proof}
For each $k\leq n-2$ there are $\binom{n}{k}$ tuples of length $k$ in $S'$. This means
\begin{align*}
    \sum_{\epsilon\in S'} \floor{\frac{n-|\epsilon|}{2}}& = \sum_{k=0}^{n-2} \binom{n}{k}\floor{\frac{n-k}{2}}
    = \sum_{k=0}^{n-2} \binom{n}{k} \frac{n-k}{2} - \frac{1}{2}\sum_{\substack{0\leq k \leq n-2 \\ n-k \text{ odd}}} \binom{n}{k}.
\end{align*}
Recall that
\begin{align*}
\sum_{k=0}^{n} \binom{n}{k} = 2^n, \quad \sum_{k=0}^n k\binom{n}{k} = n2^{n-1}.
\end{align*}
This means that
\begin{align*}
\sum_{k=0}^{n-2} \binom{n}{k}\frac{n-k}{2} = n(2^{n-2}-1/2).
\end{align*}
When $n$ is even, we have
\begin{align*}
\frac{1}{2}\sum_{\substack{0\leq k \leq n-2 \\ n-k \text{ odd}}} \binom{n}{k} &= \frac{1}{2}\sum_{k=1}^{(n-2)/2} \binom{n}{2k-1} =\frac{1}{2} (2^{n-1} - n)
\end{align*}
and when $n$ is odd, we have
\begin{align*}
\frac{1}{2}\sum_{\substack{0\leq k \leq n-2 \\ n-k \text{ odd}}} \binom{n}{k} &= \frac{1}{2}\sum_{k=0}^{(n-1)/2} \binom{n}{2k} = \frac{1}{2} \sum_{k=0}^{n-1} \binom{n}{k} 
= \frac{1}{2}(2^{n-1}-1)
\end{align*}
The claim follows.
\end{proof}

\begin{proof}[Proof of \cref{hypersurfacesTheorem}]
Fix $n\geq 2$ and consider the polytope $\Delta(n)$; this is contained in the Newton polytope of a general degree $n+2$ hypersurface in $\P^{N+1}$ where $N=2^n+n-2$. We will construct new polytopes $\Delta(n,r)\subset \R^{2^n+n-1}\times \R^r$ for any $r\leq 2^{n-2}(n-1)$. We first explain the construction for $r=1$, from which a generalization is easy. Throughout the argument, we refer to the degree of a vertex, meaning the sum of its components. If a polytope has vertices with only nonnegative components with degree at most $d$, then it is contained in the Newton polytope of a general hypersurface of degree $d$.

Pick any $\epsilon\in S'$ and consider the pairs of vectors
\begin{equation*}
e_{2n+\rho'(\epsilon)}- e_{2^n+n},\quad e_{2^n+n}.
\end{equation*}
Let $\Delta(n,1)$ be the convex hull of $\Delta(n)$ and these two vectors (notice the increase in dimension, $e_{2^n+n}$ a new basis vector). By \cref{higherDimHPTProp}, any polytope that contains this is not stably rational since $\Delta(n)$ satisfies condition (M). After a unimodular change of coordinates taking  $e_{2n+\rho'(\epsilon)}$ to $e_{2n+\rho'(\epsilon)}+e_{2^n+n}$ we get a polytope where all vertices have the same degree except for the new vertex (which have degree 1), and the vertex that contains $e_{2n+\rho'(\epsilon)}$ non-trivially, this is now
\begin{equation*}
\epsilon_1e_1+ \dots + \epsilon_n e_n + 2e_{2n+\rho'(\epsilon)}+ 2e_{2^n+n}
\end{equation*}
so has increased to degree $|\epsilon|+4$. By construction $|\epsilon|\leq n-2$ so $|\epsilon|+4\leq d$. This means that $\Delta(n,1)\subset (n+2)\Delta_{2^n+n}$ and they are of the same dimension, hence $(n+2)\Delta_{2^n+n}$ is a non-stably rational polytope. We inductively carry out the construction from $\Delta(n,r-1)$ to $\Delta(n,r)$ in the following way: Pick a vertex that has degree $\leq n-2$ that contains $e_{2n+\rho'(\epsilon)}$ non-trivially for some $\epsilon\in S'$. Increase dimension by adding elements $e_{2n+\rho'(\epsilon)}-e_{2^n+n-1+r}$ and $e_{2^n+n-1+r}$ in the same way as before. After a unimodular change of coordinates, we increase the degree of the vertex in question by 2, so the total degree of the polytope is unchanged. It is contained in the Newton polytope of a general degree $n+2$ hypersurface in $\P^{2^n+n+r}$. We can do this until all vertices have a degree at least $n-1$. Explicitly, this means that we can carry out the same construction $\floor{\frac{n-|\epsilon|}{2}}$ times for each $\epsilon\in S'$, increasing $r$ by one each time. This means that
\begin{equation*}
r\leq \sum_{\epsilon\in S'} \floor{\frac{n-|\epsilon|}{2}} = 2^{n-2}(n-1)
\end{equation*}
by \cref{sumLemma} as claimed.
\end{proof}

Similarly to \cite[Cor. 1.4]{Sch19a} the polytopes we construct for $n=3$ are contained in Newton polytopes of quintic hypersurfaces containing a 3-plane. This improves Corollary 1.4 in \cite{Sch19a} to $N=13$.

\begin{corollary}
Let $X\subset \P^{N+1}$ be a very general quintic hypersurface containing a 3-plane for $10\leq N \leq 13$. Then $X$ is smooth, unirational, and not stably rational.
\end{corollary}
\begin{proof}
The polytope of a very general quintic containing a 3-plane is as follows. Pick three vertices of $5\Delta_{N+1}$ and let $L$ be the linear space spanned by these. Suppose for simplicity that we choose $5e_1,5e_2,5e_3$. The polytope obtained as the convex hull of the lattice points in $5\Delta_{N+1}\setminus L$ is the Newton polytope of a very general quintic hypersurface of $\P^{N+1}$ containing the 3-dimensional linear space $x_0=x_4=\dots = x_{N+1}=0$. These polytopes are smooth by \cite[Cor. 8.2]{Sch19a}. 

The polytopes $\Delta(3,r)$ are easily seen to be contained in one of these for a particular choice of $L$ (possibly after a unimodular transformation), so the claim now follows for the same reason as \cite[Cor. 8.2]{Sch19a}.
\end{proof}

\subsection{Double covers}
Again using the explicit equations in \cite{Sch19a} for non-stably rational double covers of projective space, we can push the bounds further using the same ideas as above. Generalizing the previous section to double covers is much easier from the combinatorial viewpoint. The polytopes $\Delta(n)$ constructed in the last paragraph are contained in the polytope associated with a double cover of $\P^N$ ramified along a very general hypersurface of degree $2\ceil{\frac{n}{2}}+2$, so we can use these to improve the bounds for double covers more or less directly.

\begin{theorem} \label{doubleCoverThm}
Let $N\geq 3$ and write $N=n+r$ for $2^{n-1}-2\leq r\leq 2^{n}-2 + 2^{n-2}(n-1)-\floor{\frac{n}{2}}$. Then a double cover $X\to \P^N$ branched along a very general hypersurface of even degree $d\geq 2\ceil{\frac{n}{2}}+2$ is not stably rational.
\end{theorem}
\begin{proof}
First, note that the polytope of such a double cover is the following.
\begin{align*}
\Delta^{2,d,N} = \Conv{(0, de_{1},\dots, de_{N}, 2e_{N+1})}
\end{align*}
By \cref{lowerRvalueRemark}, the polytopes obtained from $2^{n-1}-2\leq r \leq 2^n-2$ are not stably rational, satisfies condition (M), and are contained in $\Delta^{2,d,N}$ for $N=n+r$ when $n\geq 6$. So, in this case, we may fix $r=2^n-2$ and consider the polytope $\Delta(n)$ constructed above.

This is a simplex, and it has the following vertex
\begin{align*}
\epsilon_1e_1+\dots + \epsilon_ne_n + 2e_{2n+\rho'(\epsilon)} = 2e_{2n+\rho'(\epsilon)}
\end{align*}
for $\epsilon=(0,\dots,0)$. Up to a unimodular transformation, the polytope $\Delta(n)$ is contained in $\Delta^{2,n+2,N}$. The polytopes $\Delta(n,s)$ constructed in the proof of \cref{hypersurfacesTheorem} are contained in $\Delta(2,n+2, N+s)$ as long as we leave the vertex $2e_{2n+\rho'(0,\dots,0)}$ untouched. This means that we can carry out the same construction for $s\leq 2^{n-2}(n-1)-\floor{\frac{n}{2}}$, proving the claim.

For cases $n=3,4$, it suffices to show that double covers of $\P^N$ branched along a very general hypersurface of degree $6$ are non-stably rational for $3\leq N \leq 14$. It is proved for $N\leq 9$ in \cite{Sch19a}. Using $\Delta(3)$, the above argument extends this to $N\leq 12$, and taking $n=4$ and $r\geq 8$, one check that there are choices of $\rho$ such that the corresponding polytopes satisfy condition (M), this proves the remaining cases.

It remains to show that double covers of $\P^N$ branched along a very general octic hypersurface is not stably rational for $19\leq N\leq 145$. The case $n=6$ covers $N\geq 36$. The remaining cases are covered in \cite{Sch19a}. 
\end{proof}

\appendix

\section{} \label{appendixA}
The goal of this appendix is to prove that certain subschemes of strictly toroidal $R$-schemes are again strictly toroidal. It is really a statement in the language of logarithmic geometry and will be proven in this context. We recall some notation but refer the reader to \cite{BN20, Kat94, Kat89} for more details using the same notation as we do here. More specifically, we prove

\begin{theorem} \label{logSubSchTheorem}
Let $\X^\dagger$ be a log regular scheme and let $\Y\subset \X$ be a closed subscheme. Denote by $\Y^\dagger$ the log scheme with log structure obtained by pullback from $\X^\dagger$.
Assume that the schematic intersection of $\Y$ with any logarithmic stratum of $\X$ is regular and dimensionally transverse (possibly empty). Then $\Y^\dagger$ is log regular.
\end{theorem}

\begin{remark}
Suppose $\X^\dagger$ is a finite type $R_0$-scheme such that the log structure is nontrivial on the special fiber. If $\X^\dagger$ is log regular, then the normalized base change to $R$ is strictly toroidal. This follows from Kato's criterion for log smoothness \cite[Prop. 3.4]{Kat89} and \cite[Prop 3.6.1]{BN20}.
\end{remark}

\cref{logSubSchTheorem} is much more general than what is needed in this paper, but we include it in full generality for future reference.

\subsection*{Notation}
For a monoid $M$, let $M^\times$ denote the invertible elements of $M$, $M^+=M\setminus M^\times$, and $M^\sharp = M/M^\times$ the characteristic monoid. If $M^\sharp=M$, we say that $M$ is sharp.

For a log scheme, $\X^\dagger$ denote the sheaf of monoids by $\M_\X$. We say that a log structure is integral, saturated, finitely generated, and fine if $\M_x$ has the same property for all $x\in \X$.

A Weil divisor $D$ in $\X$ induces a log structure on $\X$, the divisorial log structure induced by $D$. The sheaf of monoids consists of regular functions on $\X$ invertible on $\X\setminus D$. Note that $\M_\X\subset \OO_\X$ is a subsheaf, in particular, this is always integral. Throughout this appendix, we assume all log schemes to be fine and saturated (fs).

We say that $\X^\dagger$ is \emph{log regular} if for every point $x\in \X$ the local ring $\OO_{\X,x}/\mathcal{I}_{\X,x}$ is regular, where $\mathcal{I}_{\X,x}$ is the image of $\M_{\X,x}\setminus \M_{\X,x}^\times$ in $\OO_{\X,x}$, and
\begin{equation*}
        \dim{\OO_{\X,x}} = \dim{\OO_{\X,x}/\I_{\X,x}} + \dim{\M^\sharp_{\X,x}}.
\end{equation*}
Log regular schemes are normal, Cohen-Macaulay, and the log structure is always divisorial.

For a regular log scheme $\X^\dagger$ there is a fan $F= F(\X^\dagger)$ such that for every $\tau\in F$ there is a locally closed subset $E(\tau)^\circ\subset \X$ such that $E(\tau)^\circ, \tau\in F$ forms a stratification of $\X$ (see \cite[Section 3.2]{BN20} for details).

Suppose that the log structure on $\X^\dagger$ is divisorial induced by a divisor $D$. Let $\locdiv{D}{\X}$ be the sheaf of effective Cartier divisors supported in $D$. A section $s\in \Gamma(U,\M_\X)\subset \OO_{\X}(U)$ is not a zero divisor and defines a Cartier divisor on $U$ with support on $D$, unique up to multiplication with elements of $\OO_{\X}(U)^\times$. This induces an isomorphism of sheaves $M_\X^\sharp \simeq \locdiv{D}{\X}$ \cite[III.1.6.3]{Ogu18}.

\subsection*{Proof of theorem}
For the rest of this section, we let $R_0=\KK$ and $S^\dagger$ the log scheme obtained by endowing $\Spec{R_0}$ with the divisorial log structure induced by the closed point. We let $\X^\dagger$ be a fine saturated log scheme of finite type over $S^\dagger$.

Let $D=\sum_{i\in I}E_i$ denote the divisor inducing the log structure on $\X^\dagger$ where the $E_i$ are integral closed subschemes of $\X$. The open strata $E(\tau)^\circ, \tau\in F(\X^\dagger)$ are the connected components of the subschemes
\begin{equation*}
    \bigcap_{j\in J}E_j\setminus \bigcup_{i\centernot\in J}E_i
\end{equation*}
for $\emptyset \neq J\subset I$.

Let $\Y$ be an irreducible closed subscheme of $\X$. Assume that for every open stratum $E(\tau)^\circ$ of $\X$, the schematic intersections $\Y\cap E(\tau)^\circ$ are regular and dimensionally transverse. Let $\Y^\dagger$ denote the log scheme with the underlying scheme $\Y$ and the divisorial log structure induced by the divisor $\Y\cap D$. We first show that this coincides with the log structure obtained by pulling the log structure on $\X^\dagger$.
\begin{proposition}
    The induced inclusion of log schemes
    \begin{equation*}
    i\colon \Y^\dagger \to \X^\dagger
    \end{equation*}
    is strict. That is, the pullback of the divisorial log structure on $\X$ coincides with the divisorial log structure on $\Y$ induced by the restriction of $D$.
\end{proposition}
\begin{proof}
    Since the logarithmic structure on $\X$ is divisorial, it is integral and by \cite[Cor. 1.2.11]{Ogu18} it suffices to show that we have an isomorphism of sheaves 
    \begin{equation*}
        i^{-1}\M^{\sharp}_\X \simeq \M^{\sharp}_\Y.
    \end{equation*}
    We use the identifications $\M^\sharp_\X\simeq \locdiv{D}{\X}$ and $\M^\sharp_\Y\simeq \locdiv{D\cap \Y}{\Y}$ so it suffices to show that the pullback via the inclusion map induces an isomorphism of sheaves
    \begin{equation*}
        i^{-1}\locdiv{D}{\X} \to \locdiv{D\cap \Y}{\Y}.
    \end{equation*}
    First, note that since $\Y$ meets the irreducible components of $D$ transversely, pulling back effective Cartier divisors supported on $D$ is well-defined. 
    
    We first show that the above map is surjective. Let $y\in \Y\subset \X$ be a point and consider an effective Cartier divisor $F'$ in a neighborhood of $y\in \Y$ supported on $\Y\cap D$. The map $\OO_{\X,y}\to \OO_{\Y,y}$ is surjective, so we denote by $\overline{x}$ the image of $x\in \OO_{\X,y}$ via this map.
    
    Since $\Y$ meets the open strata smoothly and dimensionally transverse, the divisor $\Y\cap D$ consists of components $\Y\cap E_i$ where $D=\sum_{i\in I}E_i$. Thus, we can write $F'=\sum_{j\in J} n_j(\Y\cap E_j)$ for non-negative integers $n_j$. Let $F=\sum_{j\in J}n_jE_j$. It is a divisor in $\X$ and satisfies $\Y\cap F=F'$. We claim that this is Cartier locally around $y$. It suffices to show that $F$ is defined by a principal ideal, generated by a nonzero divisor in the local ring $\OO_{\X,y}$. Let $I=(g_1,\dots,g_r)$ be the ideal in $\OO_{\X,y}$ that defines $F$ and $(\overline{f})$ the ideal in $\OO_{\Y,y}$ that defines $F'$. Since $F'\cap \Y = F$ the image of $I$ generates the ideal $(\overline{f})$. In particular, we can write $\overline{g_i}=\overline{u}_i\overline{f}$ and 
    \begin{equation*}
        \overline{f}=\sum \overline{h}_i\overline{g}_i=\sum \overline{h}_i\overline{u}_i\overline{f}
    \end{equation*}
    for some $\overline{h_i},\overline{u_i}\in \OO_{\Y,y}$. Since $\overline{f}$ is not a zero divisor, we have $1=\sum \overline{u_i}\overline{h_i}$ and so not all $\overline{u_i}$ can be contained in the maximal ideal, consequently some $\overline{u_i}$ must be invertible. We may assume that $\overline{u_1}$ is invertible, so $(\overline{g_1})=(\overline{f})$. Consider the principal ideal $(g_1)$ in $\OO_{\X,y}$ defined by an element $g_1$ that is mapped to $\overline{g_1}$. Since $\X^\dagger$ is log regular, $\X$ is normal, so it defines an effective Cartier divisor that restricts to $F'$ and contains $F$. This means that we can write it on the form $F + E$ where $E$ restricts to $0$ on $\Y$. Since $F$ and $E$ are subschemes, and $F+E$ is an effective Cartier divisor, it follows from \cite[\href{https://stacks.math.columbia.edu/tag/07ZV}{Tag 07ZV}]{stacks-project}
 that $F$ and $E$ are both effective Cartier divisors. Thus, $F$ is an effective Cartier divisor supported on $D$ that restricts to $F'$, and this shows that the pullback map is surjective.

    To see that the map is injective, let $D_1$ and $D_2$ be two effective Cartier divisors in $\X$, locally around $y$, supported in $D$. Suppose that the schematic intersections $\Y\cap D_1$ and $\Y\cap D_2$ are equal. We claim that $D_2=D_1$ as subschemes of $\X$. Suppose, for contradiction, that $D_2\neq D_1$. Write $D=\sum_{i\in I}E_i$ as before, where $E_i$ is a reduced irreducible subscheme of codimension 1. Since $\Y$ meets the open strata transversely, the schematic intersections $\Y\cap E_i$ are regular and of codimension 1 in $\Y$. It suffices to show that if two subschemes $E_i\cap \Y$ and $E_j\cap \Y$ are equal in $\Y$, then $E_i=E_j$ in $\X$. Suppose, therefore, that $D_1$ and $D_2$ are irreducible reduced closed subschemes of codimension 1 supported on $D$ that are not equal, in particular, $\codim_\X(D_1\cap D_2)=2$.
    Since $\Y$ meets the strata transversely, we have 
    \begin{equation*}
       \codim_{\X}(\Y\cap D_1\cap D_2)=\codim_\X(\Y)+2
    \end{equation*}
    and
    \begin{equation*}
        \codim_\X(\Y\cap D_1)=1+\codim_\X(\Y) =\codim_\X(\Y\cap D_2).
    \end{equation*}
    But since $\Y\cap D_1=\Y\cap D_2$ we also have
    \begin{equation*}
        1+\codim_\X(\Y)=\codim_\X(\Y\cap D_1\cap \Y\cap D_2)=\codim_\X(\Y\cap D_2\cap D_1),
    \end{equation*}
    a contradiction.
\end{proof}
\begin{corollary} \label{sharp_cor}
    Under the above assumptions, we have for every point $y\in \Y$ that
    \begin{equation*}
        \dim(\M^\sharp_{\X,y})= \dim(\M^\sharp_{\Y,y})
    \end{equation*}
\end{corollary}

\begin{theorem}
  $\Y^\dagger$ is a regular log scheme.
\end{theorem}
\begin{proof}
First, note that $\Y$ is fs since it is a closed subscheme of $\X$, and the log structure is the pullback of the log structure on $\X$.

For any point $y\in \Y\subset \X$ we have $y\in E(\tau)^\circ$ for some $\tau\in F(\X^\dagger)$. The ideal $\I_{\X,y}$ is the image of the ideal sheaf of $E(\tau)^\circ$ in the local ring at $y$, and in particular ${\OO_{\X,y}/\I_{\X,y}}$ is the local ring of the structure sheaf of $E(\tau)^\circ$ at $y$. 
Since $\Y$ meets the strata of $\X^\dagger$ smoothly and transversely, the logarithmic stratification of $\Y^\dagger$ is given by the intersections $\Y\cap E(\tau)^\circ$. In particular the ring $\OO_{\Y,y}/\I_{\Y,y}$ is the local ring of the structure sheaf of $\Y\cap E(\tau)^\circ$. Since this intersection is regular, this ring is regular for every $y\in \Y$. It remains to show that
\begin{equation*}
    \dim{\OO_{\Y,y}} = \dim{\OO_{\Y,y}/\I_{\Y,y}} + \dim{\M^\sharp_{\Y,y}}
\end{equation*}
By \cref{sharp_cor} it suffices to show
\begin{equation*}
    \dim{\OO_{\Y,y}} - \dim{\OO_{\Y,y}/\I_{\Y,y}}=\dim{\OO_{\X,y}} - \dim{\OO_{\X,y}/\I_{\X,y}}.
\end{equation*}
Note that any log regular scheme is Cohen-Macaulay (and normal) \cite[Theorem 4.2]{Kat94}. Now let $\J_{\Y}$ denote the ideal sheaf of $\Y$ and $\J_{\Y,y}$ the stalk in $\OO_{\X,y}$, since $\Y$ is irreducible, this is a prime ideal. The following properties follow from the fact that $\X$ is Cohen-Macaulay and that $\Y$ meets the open strata $E(\tau)^\circ$ smoothly and transversely. 
\begin{align*}
\dim \OO_{\Y,y} &= \dim \OO_{\X,y}/\J_{\Y,y} = \dim{\OO_{\X,y}} - \htt{\J_{\Y,y}} \\
\htt{\J_{\Y,y}} &= \codim_\X{\Y} \\
\dim{\OO_{\X,y}/\I_{\X,y}} &= \codim_{E(\tau)^\circ} (\overline{\set{y}}\cap {E(\tau)^\circ}) \\
\dim{\OO_{\Y,y}/\I_{\Y,y}} &= \codim_{E(\tau)^\circ \cap \Y} (\overline{\set{y}} \cap E(\tau)^\circ \cap \Y)
\end{align*}
Here $E(\tau)^\circ$ is the unique open stratum that contains the point $y$ and $\overline{\set{y}}$ means the closure of $y$ in $\X$. Since $E(\tau)^\circ$ and $\Y\cap E(\tau)^\circ$ are regular, we can write the last one as
\begin{align*}
    \codim_{E(\tau)^\circ \cap \Y} (\overline{\set{y}}\cap E(\tau)^\circ\cap \Y) & = \dim{(E(\tau)^\circ\cap \Y)} - \dim{(\overline{\set{y}}\cap E(\tau)^\circ)} \\
    &= \dim{(\Y)}+ \dim{(E(\tau)^\circ)}  \\
    &~~~~~ -\dim{(\X)} - \dim{(\overline{\set{y}}\cap E(\tau)^\circ)} \\
    &= \codim_{E(\tau)^\circ}({\overline{\set{y}}\cap E(\tau)^\circ}) - \codim_\X(\Y)
\end{align*}
Where we have used that $\overline{\set{y}}\subset \Y$ so
\begin{equation*}
    \dim{(\overline{\set{y}}\cap E(\tau)^\circ\cap \Y)}= \dim{(\overline{\set{y}}\cap E(\tau)^\circ)}
\end{equation*}
Putting these identities together finishes the proof.
\end{proof}

\bibliographystyle{alpha}
\bibliography{main.bib}

\end{document}